\pdfoutput=1
\documentclass{amsart}
\usepackage{amsmath, amssymb, amsthm, amscd}
\usepackage{mathrsfs}
\usepackage{calc}
\usepackage{graphicx}
\usepackage{listings}
\usepackage{pdflscape}
\usepackage{url}
\usepackage{framed}
\usepackage{algpseudocode}
\usepackage{todonotes} 
\usepackage[all]{xy}
\newdir{ >}{{}*!/-10pt/@{>}}
\newdir^{ (}{{}*!/-5pt/@^{(}}
\newdir_{ (}{{}*!/-5pt/@_{(}}
\def\cgaps#1{}
\def\Cgaps#1{}

\def\undersetbrace#1\to#2{\underbrace{#2}_{#1}}                                                          
\def\oversetbrace#1\to#2{\overbrace{#2}^{#1}}
\def\AMSunderset#1\to#2{\underset{#1}{#2}}
\def\AMSoverset#1\to#2{\overset{#1}{#2}}

\def\o{\circ}
\def\X{\mathfrak X}
\def\al{\alpha}

\def\de{\delta}
\def\ep{\varepsilon}
\def\ze{\zeta}
\def\et{\eta}
\def\th{\theta}

\def\ka{\kappa}
\def\la{\lambda}
\def\rh{\rho}

\def\ph{\varphi}

\def\ps{\psi}

\def\De{\Delta}

\def\i{^{-1}}
\def\x{\times}
\def\p{\partial}
\let\on=\operatorname

\def\AMSonly#1{}

\def\Id{\on{Id}}
\def\R{\mathbb R}

\def\Imm{{\on{Imm}}}

\def\Diff{{\on{Diff}}}

\def\dist{\on{dist}}

\def\todomartin#1{}
\let\wt=\widetilde
\let\ol=\overline
\let\mf=\mathfrak

\theoremstyle{plain}
\newtheorem{thm}[subsection]{Theorem}

\newtheorem*{proposition*}{Proposition}
\newtheorem{lem}[subsection]{Lemma}
\newtheorem*{lem*}{Lemma}

\newtheorem*{cor*}{Corollary}
\newtheorem*{theorem*}{Theorem}
\theoremstyle{definition}
\newtheorem{remark}[subsection]{Remark}
\newtheorem*{remark*}{Remark}

\begin{document}
\title[Reparameterization invariant metrics on spaces of curves]
{Constructing reparameterization invariant metrics on spaces of plane curves}

\author{Martin Bauer, Martins Bruveris, Stephen Marsland, Peter W.\ Michor}

\markboth{M.\ Bauer, M.\ Bruveris, S.\ Marsland, P.W.\ Michor}{}

\begin{abstract}
Metrics on shape spaces are used to describe deformations that take one shape to another, and to define a distance between shapes. We study a family of metrics on the space of curves, which includes several recently proposed metrics, for which the metrics are characterised by mappings into vector spaces where geodesics can be easily computed. This family consists of Sobolev-type Riemannian metrics of order one on the space $\on{Imm}(S^1,\R^2)$ of parameterized plane curves and the quotient space $\on{Imm}(S^1,\R^2)/\on{Diff}(S^1)$ of unparameterized curves. For the space of open parameterized curves we find an explicit formula for the geodesic distance and show that the sectional curvatures vanish on the space of parameterized open curves and are non-negative on the space of unparameterized open curves. For one particular metric we provide a numerical algorithm that computes geodesics between unparameterized, closed curves, making use of a constrained formulation that is 
implemented 
numerically using the RATTLE algorithm. We illustrate the algorithm with some numerical tests between shapes.
\end{abstract}

\maketitle

\section{Introduction}
The mathematical analysis of shape has been the focus of intense research interest in recent years, not least because of applications in image analysis and computer vision, where methods based on geodesic active contours or `snakes' are used for segmentation, tracking and object recognition \cite{Sundaramoorthi2011, Sundaramoorthi2007}. Another source of applications is biomedical image analysis, where the study and comparison of shapes form a large part of the field of computational anatomy \cite{Glaunes2008, GrMi1998}.

A key problem in shape analysis is to define a distance function between shapes that can measure similarity in a computationally feasible way and act as the basis for object classification. One way to arrive at such a distance function is to equip the space of shapes with a Riemannian metric, which allows the lengths of paths between shapes to be measured. The distance between two shapes can then be defined as the length of the shortest path connecting them.

For the purposes of this paper we consider shapes to be smooth plane curves (open or closed) modulo smooth reparameterizations. A slightly narrower definition would be to define a shape as the outline of a smooth, simply connected domain in the plane; this definition excludes objects like the figure of eight, which we allow in our definition. Mathematically, a shape is represented by a smooth curve $c: S^1 \to \R^2$. To curves $c, d$ represent the same shape
if there exists a reparameterization map $\ph \in \on{Diff}(S^1)$ (that is,  the group of smooth, orientation preserving, invertible maps $\ph:S^1\to S^1$ from the circle onto itself) such that one curve is a reparameterization of the other, i.e., $c = d\circ\ph$. 

We will work with the class of regular or immersed curves $c$. A curve $c$ is regular if it has a non-vanishing tangent, i.e., $c'(\theta)\neq0$.
The diffeomorphism group $\on{Diff}(S^1)$ acts on the space of immersed curves
\[
\on{Imm}(S^1,\R^2):=\left\{c\in C^\infty(S^1,\R^2)|\;c'(\theta)\neq 0\right\}\,,
\]
from the right via $(\ph, c) \mapsto c\circ\ph$. Using this setting we can identify a shape with an equivalence class $[c]$, that is an element of the quotient space $\on{Imm}(S^1,\R^2)/\on{Diff}(S^1)$. This quotient
is the shape space of immersed closed curves modulo reparameterizations and will be denoted by $\mathcal S$. Similarly, one defines the space $\on{Imm}([0,2\pi],\R^2)$ of parameterized open curves and the space $\mathcal S_{\mathrm{open}}$ of open shapes.

To arrive at a distance function on shape space requires two steps. First, we define a Riemannian metric (i.e., an inner product measuring the length of infinitesimal deformations of a curve) on the space of immersed curves and compute geodesics on this space. The deformations $h, k$ are represented by vector fields along the curve $c$ and the inner product, which depends on the curve, is denoted by  $G_c(h,k)$. 
If the metric is invariant under the action of the reparameterization group $\on{Diff}(S^1)$, then it induces a Riemannian metric on shape space $\mathcal S$, which in turn gives rise to the geodesic distance function. The second step is to find the right representatives $c, d$ of the equivalence classes $[c]$ and $[d]$, such that the geodesic distance $\on{dist}^{\on{Imm}}(c, d)$ coincides with $\on{dist}^{\mathcal S}([c], [d])$.

\subsection{Shape metrics and related work\label{sec:lit}}

The simplest reparameterization invariant metric on $\Imm(S^1,\R^2)$ is the $L^2$-metric
\[
G_c(h,k) = \int_{S^1} \langle h, k \rangle \,ds\,,
\]
where we integrate over arc-length, $ds = |c'(\theta)| d\theta$. However, the geodesic distance induced by this metric vanishes, i.e., the distance between any two shapes is 0, which renders it unsuitable for shape analysis \cite{Michor102}.

One way to overcome this is to add terms involving higher derivatives of $h$ and $k$ to the metric, such as:
\[
G_c(h,k) = \int_{S^1} \langle h, k \rangle + A \langle D_s h, D_s k \rangle \,ds \,,
\]
where $D_sh = \tfrac 1{|c'|} h'$ denotes the arc-length derivative of $h$, and $A>0$ is a constant. This leads to the class of Sobolev-type metrics, which were independently introduced in \cite{Charpiat2007, Michor107, Sundaramoorthi2007} and studied further in \cite{Michor119, Shah2006}.

Another family of metrics, the almost local metrics \cite{Michor118,Bauer2012b,Michor98}, prevent the geodesic distance from vanishing by introducing a weight function in the integral. Examples of weight functions involving the curvature or length are $w(\theta) = 1+A \ka(\theta)^2$ and $w(\theta) = \ell(c)\i$.

Sobolev-type metrics of arbitrary order were studied in \cite{Michor125, Mennucci2008, Michor107}. Although they are a natural generalization of the $L^2$-metric from a theoretical point of view, their numerical treatment is rather involved, mainly because the geodesic equation of a Sobolev-type metric of order $k$ is a nonlinear PDE of order $2k$. Interestingly, there are special cases of first order metrics for which the geodesic equation admits explicit solutions \cite{Sundaramoorthi2011, Michor111}. Apart from these special cases, there have been some attempts to solve the geodesic equation directly for Sobolev-type metrics of order 1 for curves \cite{Mio2007} and surfaces \cite{MBMB2011}. Metrics of higher order, on the other hand, are still practically untouched. 

To avoid dealing with PDEs of high order, one can restrict one's attention to Sobolev-type metrics of order one, in particular to the family of metrics known as `elastic metrics' and studied in \cite{Mio2007}, which are of the form
\begin{equation}\label{ab-metric}
 G^{a,b}_c (h,k)= \int_{S^1} a^2\langle D_sh,n\rangle\langle D_sk,n\rangle+b^2 \langle D_sh,v\rangle \langle D_sk,v\rangle\,ds\,,
\end{equation}
with constants $a,b \in \R^+$ and with $v$ and $n$ denoting the unit tangent and normal vectors to $c$. The term involving the normal vector can be seen as measuring the bending of the curve $c$ under the deformation $h$, while the derivative of $h$ in the tangential direction measures the stretching of $c$.

Following ideas of \cite{Younes1998}, it was shown in \cite{Michor111} that it is possible to find explicit formulas for geodesics of the scale-invariant version of the elastic metric with $b^2=a^2$ on the space of parameterized curves modulo translations, rotations and scaling. To achieve this a curve $c$ was represented by the square-root of its velocity vector, $\sqrt{c'}$, with $c'$ being interpreted as a complex number. In this representation the $G^{1,1}$-metric has a particularly simple form that allows for explicit formulas for geodesics between two curves and the length of the geodesic.

A similar motivation underlies the introduction of the square root velocity transform (SRVT) in \cite{Srivastava2011}. The SRVT is a transformation that maps a curve $c$ to $R(c) = {c'}/{\sqrt{|c'|}}$ with the effect that the $G^{a,b}$-metric simplifies for $4b^2=a^2$. While not allowing explicit formulas for geodesics to be written down, this map greatly simplifies their numerical computation.

Related to the SRVT is the $Q$-transform, introduced in \cite{Kurtek2011a}, which maps a curve $c$ to $Q(c) = \sqrt{|c'|}c$. This transform generalizes easily to surfaces, but does present theoretical difficulties, as explained in Section \ref{sec:qtrans}.

There are other ways to define a Riemannian metric on shape space such as  large deformation diffeomorphic metric mapping \cite{Glaunes2008, Cotter2012}, where a Riemannian metric is induced from the diffeomorphism group $\on{Diff}(\R^2)$ of the ambient space, or the use of conformal welding \cite{Sharon06} to represent shapes as diffeomorphisms of the circle.

\subsection{The reparameterization group $\on{Diff}(S^1)$}

In this article we will consider the group of smooth, orientation preserving, invertible maps $\ph:S^1\to S^1$ from the circle onto itself. This group is sometimes denoted by $\on{Diff}^+(S^1)$. 
To shorten notation we will write $\on{Diff}(S^1)$ instead.

After equipping the manifold $\on{Imm}(S^1, \R^2)$ with a Riemannian metric, there are various ways to perform the minimization over the reparameterization group. The problem is challenging because of the nonlinear nature of $\on{Diff}(S^1)$ and the fact that all the spaces involved are infinite-dimensional.
One approach is to replace the infinite-dimensional space $\on{Diff}(S^1)$ by a finite-dimensional one and to perform the minimization over this smaller space. A possible choice for the smaller space, used in \cite{Srivastava2011}, are diffeomorphisms $\ph$ for whom the Fourier series of $\sqrt{\ph'}$ is truncated at a fixed length.

In \cite{Cotter2012} elements of $\on{Diff}(S^1)$ were generated as flows of vector fields. The advantage of using vector fields is that they form a linear space, which enables the use of gradient-based optimization algorithms. However, due to the regularization terms needed to ensure convergence of the algorithm, the computed distance failed to be symmetric. An alternative is given in \cite{Sundaramoorthi2011}, where the authors describe an iterative procedure directly on $\on{Diff}(S^1)$ based on geometric considerations. We will expand on this minimization scheme in Section~\ref{sec:numerics}.

\subsection{Overview of the paper}

This paper arose from the observation that the SRVT \cite{Srivastava2011} and the $Q$-transform \cite{Kurtek2011b, Kurtek2011a, Mani2010}  
are both special cases of a general mathod for constructing reparameterization invariant metrics on spaces of curves.
We show in Section~\ref{sec:relating} how to use this method to obtain the elastic metric with general parameters $4b^2 \geq a^2$ as well as Sobolev metrics of higher order. The idea is to construct a transformation that maps $\on{Imm}(S^1,\R^2)$ isometrically to a submanifold of a flat Riemannian manifold.

In the following we exploit this representation of the space $\on{Imm}(S^1,\R^2)$ to study in particular the family of elastic metrics and their mathematical properties. We chose to focus on these metrics because they -- and especially the one induced by the SRVT -- are currently used in shape analysis.

The embedding of $\on{Imm}(S^1,\R^2)$ into a flat space also allows us to discretize the geodesic equation in a very geometric way. After choosing a discretization of the curve, the infinite-dimensional Riemannian manifold becomes a finite dimensional constrained submanifold of the Euclidean space. In Section~\ref{sec:numerics} we describe this procedure~-- which can be applied to all metrics constructed in Section~\ref{sec:relating}~-- in the particular case of the metric arising from the SRVT. The resulting shooting method to solve the geodesic boundary value problem complements the path-straightening approach used in \cite{Samir2012}.

When computing the distance between unparameterized curves (Section~\ref{sec:unparam}) it is necessary to minimize over the reparameterization group. We give a rigorous interpretation of the iterative procedure proposed in~\cite{Sundaramoorthi2011} as a gradient descent on $\on{Diff}(S^1)$ with respect to a specific Riemannian metric. Reparametrization denotes the sliding of the points representing the curve along it. Numerical problems arise when points are compressed too closely together or stretched too far apart. We will argue that this behaviour is not introduced by the discretization, but is a manifestation of geodesic incompleteness of the underlying Riemannian metric: there exist shapes $[c], [d] \in \mathcal S$ with no geodesic connecting them. We conclude the paper by discussing the repercussions of the incompleteness on applications to shape analysis.

\section{Notation and background material}

\subsection{Notation}

In this section we introduce some notation that we will use throughout the article. 
We study the spaces of open and closed, regular, plane curves 
\begin{align*}
 \on{Imm}([0,2\pi],\R^2)&= \left\{c\in C^\infty([0,2\pi],\R^2)|\;c'(\theta)\neq 0\right\}\,,\\
 \on{Imm}(S^1,\R^2)&=\left\{c\in C^\infty(S^1,\R^2)|\;c'(\theta)\neq 0\right\}\,.
\end{align*}
The elastic metric, that we will consider in this article, has a kernel on the manifold of immersions, which consists exactly of the constant vector fields. 
Therefore we introduce the quotient space of immersions modulo translations $\Imm(M,\R^2)/\on{Tra}$, for $M$ either $S^1$ or $[0,2\pi]$.

We will also need the  diffeomorphism group of both $[0,2\pi]$ and $S^1$, i.e., the groups of smooth, orientation preserving, invertible maps onto itself: 
\begin{align*}
 \on{Diff}([0,2\pi])&= \left\{\ph \in C^\infty([0,2\pi],[0,2\pi]) |\;\ph'(\theta)>0\right\}\,,\\
 \on{Diff}(S^1)&=\left\{\ph\in C^\infty(S^1,S^1)|\;\ph'(\theta)>0\right\}\,.
\end{align*}
These groups are sometimes denoted by $\on{Diff}^+(S^1)$ and $\on{Diff}^+([0,2\pi])$ respectively. 
To shorten notation we will write $\on{Diff}(S^1)$ and $\on{Diff}([0,2\pi])$ instead.

For a  curve $c$ we denote by $v = c'/|c'|$ the unit length tangent vector and by $n$ the unit length normal vector, obtained from $v$ through a rotation by $\tfrac \pi2$. 
The arc-length derivative of a function $f$ along the curve is denoted by $D_s f = \frac{1}{|c'|} f'$ and $ds = |c'| d\theta$ is integration with respect to arc-length. 
The turning angle $\al$ is defined through the relation $v = (\cos \al, \sin \al)$ and the curvature by $\ka=\langle D_s v, n\rangle$.

\subsection{Variational formulas}
In this section we will compute some variational formulae for the quantities that have been introduced in the previous 
section, i.e., we compute how these terms change if we vary the underlying curve $c$.  
For a smooth map $F$ from $\on{Imm}(S^1,\R^2)$ to any convenient vector space  we denote by
\[
dF(c).h = D_{c,h}F = \left.\frac{d}{d t}\right|_{t=0}F(c + th)  = 
\left.\frac{d}{dt}\right|_{t=0}F(\wt c(t,\quad))
\]
the variation  in the direction $h$, 
where $\wt c:\mathbb R\x M\to \mathbb R^2$ is any smooth variation with $\wt c(0,\th)= c(\th)$ and 
$\p_t|_0 \wt c(t,\th)=h(\th)$ for all $\th$.
Examples of maps $F$ include $v$, $n$, $\al$, $|c'|$, $\ka$.

The following  formulae will be used repeatedly throughout the article.
\begin{lem}\label{variational}
The first variations of the unit tangent vector $v$, the normal vector 
$n$, the length element $|c'|$ and the curvature $\ka$ are given by:
\begin{align}
\label{var_v}
D_{c,h} v &= \langle D_s h, n \rangle n \\
\label{var_n}
D_{c,h} n &= -\langle D_s h, n \rangle v \\
D_{c,h} |c'| &= \langle D_s h, v \rangle |c'| 
\label{var_length}\\
D_{c,h} \ka &= \langle D_s^2 h, n \rangle - 2\ka \langle D_s h, v \rangle\,.
\label{var_ka}
\end{align}
\end{lem}
\begin{proof}
The proof of these formulae can be found for example in  \cite{Michor107}. 
\end{proof}

\section{Constructing reparameterization invariant metrics via shape transformations\label{sec:relating}}

In this section we describe a general method of constructing reparameterization invariant metrics on the space of plane curves. This method will include the metrics studied by Srivastava et al. \cite{Srivastava2011}, Mio et al. \cite{Mio2007}, the analogue for curves of the surface metric from Kurtek et al. \cite{Kurtek2011b, Kurtek2011a} and Younes et al. \cite{Michor111} as special cases.

Let us consider a general transform $F: \Imm(S^1,\R^2)\rightarrow C^\infty(S^1,\R^n)$ mapping plane curves to $\R^n$-valued functions for some $n \in \mathbb N$. On $C^\infty(S^1,\R^n)$ we consider the following $L^2$-metric,
\begin{equation}
\label{eq:l2_flat}
G^{L^2}_q(h, h) = \int_{S^1} |h(\th)|^2 d \th\,,
\end{equation}
for $q \in C^\infty(S^1,\R^n)$, $h \in T_q C^\infty(S^1,\R^n)$. This is a particularly simple Riemannian metric: it does not depend on the basepoint $q$; it is a flat metric in the sense of Riemannian geometry with geodesic distance  given by the $L^2$-norm
\[
\on{dist}^{L^2}(q_0, q_1)^2 = \int_{S^1} |q_0(\th) - q_1(\th)|^2 d \th\,.
\]
We will define the metric $G^F$ on $\on{Imm}(S^1,\R^2)$ to be the pullback of the $L^2$-metric under the transform $F$:
\[
G_c^F(h,h) = G_{F(c)}^{L^2}(D_{c,h}F, D_{c,h} F) = \int_{S^1} \left| D_{c,h} F \right|^2 d \th\,.
\]
To obtain a re\-pa\-ra\-mete\-ri\-sa\-tion invariant metric $G^F$ we need the transform $F$ to satisfy an equivariance property.

\begin{thm}
\label{thm:reparam}
If $F: \Imm(S^1,\R^2)\rightarrow C^\infty(S^1,\R^n)$ satisfies
\begin{equation}
\label{equivariance}
 F(c\circ\ph)=\sqrt{\ph'}F(c)\circ\ph,
\end{equation}
with $c \in \Imm(S^1, \R^2)$ and $\ph \in \Diff(S^1)$ and if $F$ is infinitesimally injective, i.e., $D_{c,\cdot} F$ is injective for all $c$, then $G^F$ is a Riemannian metric on $\on{Imm}(S^1,\R^2)$ that is invariant under the reparameterization group $\Diff(S^1)$.
\end{thm}

\begin{proof}
The infinitesimal version of the equivariance property is
\[ D_{c\circ\ph,h\circ\ph}F = \sqrt{\ph'} \left(D_{c,h}F\right) \circ \ph \]
and thus we see that
\begin{align*}
 G_{c\circ\ph}^F(h\circ\ph,h\circ\ph) &= \int_{S^1} |D_{c\circ \ph, h \circ \ph} F|^2\,d\theta \\
&= \int_{S^1} \left| \left(D_{c, h} F  \right) \circ\ph\right|^2 |\ph'| \,d\theta = G_c^F(h,h)\,,
\end{align*}
i.e., the metric $G^F$ is invariant under $\Diff(S^1)$.
\end{proof}
\begin{remark}\label{Rem:equivariance}
Examples of such transforms can be constructed in the following way: Take a smooth function $f \in C^\infty(\R^{2m}, \R^n)$ and define the transform $F$ as
\begin{equation}
\label{f_transform}
F(c) = \sqrt{|c'|}\, f \o (c, D_s c, \ldots, D_s^{m-1} c)\,.
\end{equation}
To see that $F$ indeed satisfies the required invariance property, note that
\[
\sqrt{|(c\o\ph)'|} = \left(\sqrt{|c'|}\o\ph\right) \sqrt{\ph'}
\]
and that the arc-length derivative is equivariant with respect to reparameterizations. If we write $D_c$ for the arc-length derivative to emphasize the dependence on the curve $c$, the equivariance property reads as $D_{c\o\ph}(h\o\ph) = (D_c h) \o \ph$. Thus
\begin{align*}
F(c\o\ph) &=
\left(\sqrt{|c'|}\o\ph\right) \sqrt{\ph'} \, f \o \left(c \o \ph, (D_s c) \o \ph, \ldots\right)\\
&= \sqrt{\ph'}\, \left(\sqrt{|c'|}\, f \o (c, D_s c, \ldots)\right) \o \ph \,.
\end{align*}
The order of the induced metric will correspond to the order of the highest derivative appearing in the definition of the transform $F$.
\end{remark}
\begin{remark}
Often the transform $F$ will not be infinitesimally injective, but will have a kernel that consists of the constant vector fields. In this case we can regard $F$ as an infinitesimally injective map 
$F : \Imm(S^1,\R^2)/\on{Tra} \to C^\infty(S^1,\R^n)$
and it induces a Riemannian metric on the space $\Imm(S^1,\R^2)/\on{Tra}$ 
of parameterized curves modulo translations.
\end{remark}

We note that the same construction can also be applied for metrics on the space $\on{Imm}([0,2\pi], \R^2)$ of open curves.

\subsection{The square root velocity transform}\label{sec:Rtransform_original}

Our first example is the so-called square root velocity transform (SRVT), as introduced in \cite{Srivastava2011}:
\begin{align*}
&R:\;\Imm(S^1,\mathbb R^2)/\on{Tra}\to  C^\infty(S^1,\mathbb R^2) \\
&R(c) = |c'|^{1/2} v\enspace.
\end{align*}
The metric induced by this transform is
\[
G^{1,1/2}_c (h,h)= \int_{S^1} \langle D_sh,n\rangle^2+\tfrac 14 \langle D_sh,v\rangle^2\,ds\,,
\]
which is a special case of the family in \eqref{ab-metric} of elastic metrics. The transform $R$ can be written in the form \eqref{f_transform} with the function $f : \R^4 \to \R^2$, $f(x_1,x_2) = x_2$. To see this, note that $v = D_s c$. We will consider a generalization of this transform that allows us to represent metrics of the family in \eqref{ab-metric} for arbitrary parameters $a,b$ in Section \ref{sec:openR}.

\subsection{The transform of Younes et al. \cite{Michor111}}

The method applied in Younes et al. \cite{Michor111} to study a Sobolev-type metric on curves also fits within the setup described in this paper. The basic mapping considered in \cite{Michor111} is given by
\begin{align*}
&\Phi: C^\infty([0,2\pi],\R^2\setminus 0) \to \Imm([0,2\pi],\mathbb R^2)/\on{Tra} \\
&\Phi(q)(\theta) = \frac{1}{2}\int_0^\theta  \,\begin{pmatrix} q_1^2(u)-q_2^2(u)  \\2 q_1(u) q_2(u) \end{pmatrix}\,du\,.
\end{align*}
It has the property that it pulls back the Sobolev metric 
\[
G_c(h,h) = \int_0^{2\pi} |D_s h|^2 ds
\]
to the flat $L^2$-metric on the space $C^\infty([0,2\pi],\R^2\setminus 0)$. 
Here $\Imm([0,2\pi], \mathbb R^2)$ denotes the manifold of open curves. 

To fit this mapping into our framework we consider the inverse of $\Phi$:
\begin{align*}
&\Phi\i:\Imm([0,2\pi], \mathbb R^2)/\on{Tra} \to C^\infty([0, 2\pi], \R^2 \setminus 0) \\
&\Phi\i(c) = \sqrt{|c'|} \begin{pmatrix}
\cos \tfrac \al 2 \\
\sin \tfrac \al 2
\end{pmatrix}\, ,
\end{align*}
with $\al$ denoting the turning angle. If one wants to study closed curves instead of open curves, one needs to restrict the above mappings to the corresponding subspaces. 

\begin{thm}[From \cite{Michor111}]
The transform $\Phi\i$ induces the Sobolev metric of order one
\[
G_c(h,h)= \int_{0}^{2\pi} |D_s h|^2 \,ds
\]
on the space of parameterized curves modulo translations, i.e., the $a$-$b$ metric with $a=b=1$.
\end{thm}

\subsection{The $Q$-transform}
\label{sec:qtrans}

Another member in this family of shape transformations is the $Q$-transform, which was introduced in Mani et al. \cite{Mani2010} for curves and in Kurtek et al. \cite{Kurtek2011b, Kurtek2011a} for surfaces. 

Following \cite{Mani2010}, let us define for a curve $c\in\Imm(S^1,\R^2)$ the transform
\begin{align*}
&Q:\;\Imm(S^1,\mathbb R^2)\to  C^\infty(S^1,\mathbb R^2) \\
&Q(c)= \sqrt{|c'|}\, c\,.
\end{align*}

This transform is a special case of the family in \eqref{f_transform}, with the function $f:\R^2\to\R^2$, $f(x) = x$.

\begin{thm}\label{thm:qtrans}
The $Q$-transform induces a reparameterization invariant metric on $\Imm(S^1,\R^2)$, given by
\begin{align}
\label{eq:gqmetric}
 G_c(h,h)
= \int_{S^1} \big\langle h + \tfrac{1}{2} \langle D_s h, v \rangle c,
               h + \tfrac{1}{2} \langle D_s h,v \rangle c \big\rangle\, ds\,.
\end{align}
\end{thm}
\begin{remark}
Note that the $Q$-transform induces a metric on  $\Imm(S^1,\R^2)$, the space of all parameterized curves. This is in contrast to the other transforms discussed in this section, which induce metrics on the space $\Imm(S^1,\R^2)/\on{Tra}$ of curves modulo translations.
\end{remark}
\begin{proof}
Using Lemma~\ref{variational} we can compute the variation of $Q$. It is given by
\begin{equation}
\label{eq:TQ}
D_{c,h} Q =  \left( h + \tfrac{1}{2} \langle D_s h, v \rangle c \right) \sqrt{|c'|} \,.
\end{equation}
This shows that the metric has the form \eqref{eq:gqmetric}.

It remains to show that the map $h \mapsto D_{c,h}Q$ is injective for all $c$. Since we are only considering immersions that satisfy $|c'|\neq 0$, the condition $D_{c,h} Q=0$ is equivalent to
\begin{equation}
\label{eq:TQ_zero}
\langle h', c' \rangle c + 2|c'|^2 h = 0\,.
\end{equation}
For all $\theta$, where $c(\theta)=0$, this already implies $h(\theta)=0$. On the open set $\{\theta: c(\theta)\ne 0\}$ we expand the tangent vector $h$ into a part along $c$ and a part orthogonal to it, $h = h_1 c + h_2 c^\perp$ with $c\perp c^\perp$. Then \eqref{eq:TQ_zero} can be rewritten as
\[
\langle h', c' \rangle c + 2|c'|^2 h_1 c + 2|c'|^2 h_2 c^\perp = 0\,.
\]
This implies that $h_2=0$ everywhere and we are left with
\begin{equation}
\label{eq:h1_zero}
\langle c, c' \rangle h_1' + 3|c'|^2 h_1 = 0\,. 
\end{equation}
As before, we solve for $\theta$ satisfying $\langle c(\theta), c'(\theta)\rangle=0$ and find that $h_1(\theta)=0$. It remains to study  the open set, where neither $c(\theta)$ nor $\langle c(\theta), c'(\theta)\rangle$ vanish. This open set is the union of disjoint open intervals. Denote one such interval by $(\theta_0, \theta_1)$. On this interval every solution of \eqref{eq:h1_zero} is given by
\[
h_1(\theta) = C e^{\int_{\wt\theta}^\theta\frac{-3|c'|^2}{\langle c, c'\rangle}d\theta }\,,
\]
with $\wt \theta \in (\theta_0, \theta_1)$. We are looking for a smooth solution $h_1$ on all of $S^1$ and we already know that for $\theta_0$ and $\theta_1$ the solution has to satisfy $h_1(\theta_0) = h_1(\theta_1) = 0$, because these points lie outside the open set. Thus we see that only the solution with $C=0$ can satisfy this. Therefore $h_1\equiv 0$ and with it $h\equiv 0$ on all of $S^1$.
\end{proof}

The injectivity of the map $D_{c,\cdot} Q$ is essential in order to use the metric for shape comparisons. However, we could not find a proof for this anywhere in the literature.

Questions that have a comparably easy answer for the family of $R$-transforms are much more difficult for the $Q$-transform. To our knowledge it is unknown wether $Q$ itself is injective; similarily little is known about its image  on either open or closed curves, e.g., whether the image is an open subset or a smooth submanifold of the space $C^\infty(S^1, \R^2)$. See Sections \ref{sec:openR} and  \ref{sec:closedR} for an answer to these questions in the case of the $R$-transform. Even finding a numerically efficient way to invert the $Q$-transform presents difficulties, mainly because the $Q$-transform scales the curve $c$ with the object $\sqrt{|c'|}$, which is geometrically the square-root of the volume element on the curve and multiplies tangent vectors along $c$, not the curve itself. 

\subsection{An $H^2$-type metric}\label{H2metric}
We can also use this method to construct higher order metrics, such as the following transform, which induces a second order Sobolev-type metric:
\begin{align*}
&K:\;\Imm(S^1,\mathbb R^2)/\on{Tra}\to  C^\infty(S^1,\mathbb R^3) \\
&K(c)=  \sqrt{|c'|} (v,\ka)\,.
\end{align*}
\begin{thm}
The transform $K$ induces a reparameterization invariant metric on $\Imm(S^1,\R^2)/\on{Tra}$. It is given by
\begin{align*}
G_c(h,h)= \int_{S^1}& \langle D^2_s h,n\rangle^2  -3\ka \langle D^2_s h,n\rangle \langle D_s h , v\rangle      \\&+  \langle D_sh, n \rangle^2  +
\frac14 (1 +9\ka^2)\langle D_sh, v \rangle^2
\,ds\,.
\end{align*}
\end{thm}
\begin{proof}
To calculate the formula of the pullback metric we compute the variation of $K$. Using Lemma~\ref{variational} we obtain:
\begin{equation*}
\label{eq:TK}
D_{c,h} K = \sqrt{|c'|}
\begin{pmatrix}
\frac12 \langle D_sh, v \rangle v +  \langle D_sh, n \rangle n  \\
\langle D^2_s h,n\rangle -\tfrac32 \ka \langle D_s h , v\rangle  
\end{pmatrix}\,,
\end{equation*}
from which we can deduce the formula for the pullback metric. 
The injectivity of the map $K$ and its derivative is clear, since one can reconstruct the immersion up to translations from the first two components of $K$, c.f., Section \ref{sec:imageRmap}.
\end{proof}

Note that the transform $K$ is not surjective. In fact, since it maps curves to $\R^3$-valued functions, the image will have infinite codimension. This is in contrast to the transform considered in Section \ref{sec:openR}, 
whose image is open on the space of open curves and has finite codimension for closed curves. This is the price that has to be paid for increasing the order of the metric. Second order metrics, that are induced in a similar way, 
are studied in \cite{Bauer2013a_preprint}.

\subsection{The elastic metric in Mio et al. \cite{Mio2007}}

The representation of the elastic metric used in \cite{Mio2007} can also be rephrased in the spirit of representing metrics as pull-backs of simpler metrics.  Define on $\R^2$ the Riemannian metric
\[
g_{(x,y)} = a^2 e^x dx^2 + b^2 e^x dy^2
\]
and consider  parallel to \eqref{eq:l2_flat} the $L^2$-metric on the space $C^\infty(S^1,(\R^2,g))$,
\begin{equation}
\label{eq:l2_nonflat}
G_q^{L^2(g)}(h,h) = \int_{S^1}g_{q(\th)}(h(\th), h(\th))\,d\th\,.
\end{equation}
Then the transform
\begin{align*}
&R:\;\Imm(S^1,\mathbb R^2)/\on{Tra}\to  C^\infty(S^1,\mathbb (\R^2,g)) \\
&R(c)=  (\log |c'|, \al)
\end{align*}
pulls back the $L^2$-metric to the elastic metric \eqref{ab-metric}. While in this case the metric \eqref{eq:l2_nonflat} is not as simple as the $L^2$-metric in \eqref{eq:l2_flat}, 
it is still simpler than the elastic metric, since it does not contain spatial derivatives.

\section{The $R$-transform for open curves\label{sec:openR}}

We want to generalize the SRVT, as defined in Section \ref{sec:Rtransform_original}, to
study the family of $a$-$b$ metrics \eqref{ab-metric} for arbitrary parameters $a$ and $b$. 
We will start by studying the metric on the set of open curves $\Imm([0,2\pi], \R^2)/\on{Tra}$, and then investigate the closedness conditions.

Define the $R$-transform of a plane curve $c \in \Imm([0,2\pi], \R^2)$ by
\begin{align*}
&R^{a,b}:\;\Imm([0,2\pi],\mathbb R^2)/\on{Tra}\to  C^\infty([0,2\pi],\mathbb R^3) \\
&R^{a,b}(c) = |c'|^{1/2}\left(a\begin{pmatrix} v\\ 0\end{pmatrix}
+ \sqrt{4b^2-a^2}\begin{pmatrix} 0\\ 1 \end{pmatrix}\right)\enspace.
\end{align*}
Here $a,b\in\R^+$ are positive numbers with $4b^2\geq a^2$. We will omit the parameters $a,b$ when the meaning is unambiguous. The $R$-transform maps an open plane curve to a space curve. We will see that equipping the space $C^\infty([0,2\pi],\mathbb R^3)$ with a flat $L^2$-metric, i.e., considering the vector space $C^\infty([0,2\pi],\mathbb R^3)$ with the $L^2$-inner product as a Riemannian manifold, will generate Sobolev metrics of order one via the pullback by the $R$-transform. For the choice of parameters $4b^2=a^2$ the $R$-transform reduces to the SRVT of Section \ref{sec:Rtransform_original}, as studied in \cite{Srivastava2011}. In the following theorem we will show that the metrics induced by the $R^{a,b}$-transform coincide 
with the family of $a$-$b$ metrics in \eqref{ab-metric} as introduced in \cite{Mio2007}.

First let us note the following properties of the $R$-transform.
\begin{lem}\label{lem:equivariance}
For a curve $c\in \Imm([0,2\pi],\mathbb R^2)$ and  $\ph\in\Diff([0,2\pi])$ a diffeomorphism we have:
\begin{itemize}
\item
Equivariance under reparametrizations
\[ R(c\circ\ph) = |\ph'|^{1/2} . (R(c)\circ\ph) \]
\item
Translation invariance
\[ R(c + p) = R(c) \text{ for } p\in \mathbb R^2 \]
\item
Scaling property
\[ R(\rh.c ) = \rh^{1/2}. R(c) \text{ for } \rh\in \mathbb R_{>0} \]
\item
Preservation of the $L^2$-norm under reparametrizations
\[ \int_{0}^{2\pi} |R(c\circ \ph)|^2\,d\theta = \int_{S^1} |R(c)|^2\,d\theta\]
\end{itemize}
\end{lem}

\begin{proof}
The first property follows from Remark \ref{Rem:equivariance}. The other properties can be verified by simple calculations.
\end{proof}

\begin{thm}[The pullback metric on $\Imm/\on{Tra}$]\label{thm:grmetric}
The pullback of the $L^2$-inner product on $C^{\infty}([0,2\pi],\mathbb R^3)$ 
to the manifold of immersions, $\Imm([0,2\pi],\R^2)$, by the $R^{a,b}$-transform yields the family of reparameterization invariant $a$-$b$ metrics on the space $\Imm([0,2\pi],\R^2)/{\on{Tra}}$:
\[
G_c^{a,b}(h,h)= \int_{0}^{2\pi} a^2\langle D_sh,n\rangle^2+b^2 \langle D_sh,v\rangle^2\,ds\,.
\]
\end{thm}

\begin{proof}
The pullback metric is defined via
\[ G^{a,b}_c(h,h) = \left\langle D_{c,h}R, D_{c,h}R \right\rangle_{L^2},\]
and hence we need to compute the variation of the $R$-transform. Using the formulas for the derivatives of the functions $c\mapsto |c'|^{1/2}$ and $c \mapsto v$
from Lemma \ref{variational} we obtain
\begin{align*}
 D_{c,h} R 
&= |c'|^{1/2} \begin{pmatrix} a \langle D_sh, n \rangle n + \frac a2 \langle D_sh, v \rangle v \\ \frac 12 \sqrt{4b^2-a^2} \langle D_sh, v \rangle \end{pmatrix}.
\end{align*}
Therefore the pullback metric is
\begin{align*}
G^{a,b}_c(h,h) &= \int_0^{2\pi} \bigg( \left| a \langle D_sh, n \rangle n + \frac a2 \langle D_sh, v \rangle v \right|^2 
+ \left(b^2-\frac{a^2}4\right) \langle D_sh, v \rangle^2 \bigg) |c'| \,d\theta \\
&= \int_0^{2\pi}  a^2\langle D_sh,n\rangle^2+b^2 \langle D_sh,v\rangle^2\,ds\,.
\end{align*} 
Lemma \ref{lem:equivariance} tells us that the metric $G^{a,b}$ is invariant under reparameterizations.
This can also be seen directly, since the metric is written only in terms of operations $D_s, v, n$, which are equivariant with  respect to reparameterizations. 
The kernel of the transform consists only of constant vector fields, since one can reconstruct $D_sh$ from $D_{c,h}R^{a,b}$. Thus $R^{a,b}$ induces a metric on $\Imm(S^1,\R^2)/\on{Tra}$.

\end{proof}

\subsection{Geodesic Equation}
This family of metrics can be also written in the form
\[ G_c^{a,b}(h,h) = \int_{0}^{2\pi}\langle P^{a,b}_c(h),h \rangle\,ds \]
where for each $c$ the associated pseudo-differential operator $P^{a,b}_c$ is given by
\begin{equation}
\label{eq:ass_diff_open}
\begin{aligned}
 P^{a,b}_c(h)=&
-D_s \left( a^2 \langle D_s^2h,n\rangle n + b^2 \langle D_s^2h,v\rangle v\right) 
\\&\qquad
+(\de_{2\pi}-\de_0)\big(a^2 \langle n,D_sh\rangle n +b^2 \langle v,D_sh\rangle v\big)\,.
\end{aligned}
\end{equation}
Here $\delta$ denotes the delta distribution.

If we identify the space of immersions modulo translations with the section
\[
\on{Imm}([0,2\pi],\R^2)/\on{Tra} \cong
\{ c \,:\, c(0) = 0 \}\,,
\]
then the geodesic equation for the $G^{a,b}$-metric is
\begin{equation}
\label{eq:geod_eq_open}
\begin{aligned}
D_s \left( \left(A_c c_t \right)_t
+ \frac 12 B_c(c_t,c_t) \right) &= 0 \\
\left( A_c c_t \right)_t(2\pi) + \frac 12 B_c(c_t,c_t)(2\pi) &= 0\,,
\end{aligned}
\end{equation}
with
\begin{align*}
A_c h &= a^2 \langle D_s h, n \rangle n + b^2 \langle D_s h, v \rangle v\\
B_c(h,h) &= \left( a^2 \langle D_s h, n \rangle^2 + b^2 \langle D_s h, v \rangle^2 \right) v - 2(b^2 - a^2)\langle D_s h, n \rangle \langle D_s h, v \rangle n\,.
\end{align*}
We can then integrate the equation once and use the boundary conditions to determine the constant of integration. The equation we obtain is
\[
\left(A_c c_t \right)_t + \frac 12 B_c(c_t,c_t) = 0\,.
\]

On the space of closed curves, the associated operator is the same as in \eqref{eq:ass_diff_open}, except that the terms involving the delta-distribution disappear. The geodesic equation is the same as \eqref{eq:geod_eq_open} minus the boundary conditions. 

From the results of Sect. \ref{sec:imageRmap} we can deduce an explict solution formula for the geodesic equation on the space of open curves. Given a curve $c_0$ and an initial velocity $u_0$, the solution of the geodesic equation for small time is given by
\begin{equation*}
c(t,\th) = \left(R^{a,b}\right)\i\big(R^{a,b}(c_0)+t D_{c_0,u_0} R^{a,b}\big)\,.
\end{equation*}
Even though the geodesic equation on the space of closed curves is simpler, there exists no explicit solution formula for it; see Section \ref{sec:numerics} for numerical computations of geodesics on closed curves.

\subsection{Image of the $R$-transform}\label{sec:imageRmap}

To characterize the image of the $R$-transform we note that we can reconstruct $c'$ from the first two components of $R(c)$ via
\[ c'= \frac{1}{a^2} \sqrt{R_1^2(c)+R_2^2(c)}\,\begin{pmatrix} R_1(c)  \\ R_2(c) \end{pmatrix} \]
and hence $ a |c'|^{1/2} = \sqrt{R_1^2(c)+R_2^2(c)} $. This implies that
\[ R_3(c) = \frac{\sqrt{4b^2-a^2}}{a} \sqrt{R_1^2(c)+R_2^2(c)}\,, \]
which can be written in the form
\begin{align} \label{formula_cone}
 \left(4b^2-a^2\right)\left(R_1(c)^2 + R_2(c)^2\right) = a^2 R_3(c)^2\,. 
\end{align}

Let us define the cone
\[ C^{a,b} = \left\{ q \in \R^3\,:\, (4b^2-a^2)(q_1^2 + q_2^2) = a^2q_3^2,\, q_3 > 0 \right\} \]
in $\R^3$. Then the image of the $R$-transform consists of curves, that lie in $C^{a,b}$. That is, for a curve $q \in  C^{\infty}([0,2\pi],\mathbb R^3)$ we have
\begin{align*}
 q \in \on{im} R &\Longleftrightarrow q(\theta) \in C^{a,b},\, \forall \theta \in [0,2\pi]\,, \\
\on{im} R &= C^\infty([0,2\pi], C^{a,b})\,.
\end{align*}
In the special case $4b^2=a^2$ the $R^{2a,a}$-transform has no third component and the image of the $R^{2a,a}$-transform is the open set consisting of all curves which avoid $(0,0) \in \R^2$. The latter condition arises from the requirement that the curves be regular. 

The inverse of the $R$-transform can be computed using the identity
\[ c' = \frac 1{2ab} |R| \begin{pmatrix} R_1(c) \\ R_2(c) \end{pmatrix}. \]
Therefore
\begin{align*}
&R\i : \on{im} R \to \Imm([0,2\pi], \R^2)/\on{Tra} \\
&R\i(q)(\theta) = p_0 + \frac{1}{2ab} \int_0^\theta |q(\theta)| \begin{pmatrix} q_1(\theta) \\ q_2(\theta) \end{pmatrix} d\theta\,.
\end{align*}
Of course, the inverse of the $R$-transform is defined only up to translation, which manifests itself as the freedom to choose the starting point $p_0\in \R^2$ of the integration.

The cone $C^{a,b}$ is a flat hypersurface in $\R^3$, since there is an isometric covering map from the polar coordinate domain $\{ (r,\ph)\,:\, r>0\}$ to $C^{a,b}$ given by
\[ q(r,\ph) = \left( \tfrac rm \cos(m\ph), \tfrac rm \sin(m\ph), \tfrac {\sqrt{4b^2-a^2}}{2b} r \right)\,, \]
with $m = \frac {2b}a$. The computation 
\[ (4b^2-a^2)(q_1^2 + q_2^2) = \frac {4b^2-a^2}{m^2} r^2 = a^2 q_3^2\,, \]
checks that the image of this map is the cone $C^{a,b}$. From
\begin{align*}
dq_1^2 + dq_2^2 + dq_3^2 &= \frac{1}{m^2} dr^2 + r^2 d\ph^2 + \frac{4b^2-a^2}{4b^2} dr^2 = dr^2 + r^2 d\ph^2
\end{align*}
we see that the map $q(r,\ph)$ is an isometry from the Euclidean metric in $\R^2$, which has the expression $dr^2 + r^2 d\ph^2$ in polar coordinates, to the natural metric on the cone $C^{a,b}$.

The inverse map is determined only up to a multiple of $2\pi$ and is given by
\[
\Phi(q)=\left(\begin{aligned}
r(q) &= \frac{2b}{\sqrt{4b^2-a^2}}q_3
\\
\ph(q) &= \frac{a}{2b}\left(\arctan\left(\frac{q_2}{q_1}\right) + 2k\pi\right)
\end{aligned}\right)\,
\]
with $k \in \mathbb Z$. Using the inverse map we can write the distance function on the cone
\begin{align}
\label{eq:geod_length}
\on{dist}^2(q,\ol q) &=
(x_1(q)-x_1(\ol q))^2 + (x_2(q)-x_2(\ol q))^2 \\
\nonumber
&=r(q)^2 + r(\ol q)^2 -2r(q)r(\ol q)\cos\left(\ph(q)-\ph(\ol q)\right) \\
\nonumber
&= \min_{k \in \mathbb Z} \frac{4b^2}{4b^2-a^2} \bigg(
q_3^2 + \ol q_3^2 -{} \\
\nonumber
&\qquad{}
-2q_3 \ol q_3 
\cos\left( \frac a{2b} \left( \arctan \frac{q_2}{q_1} -
\arctan\frac{\ol q_2}{\ol q_1} \right) + \frac ab k \pi \right) \bigg)\,.
\end{align}
The minimum appears because the angle is only determined up to a multiple of $2\pi$.

\begin{thm}
\label{thm:open_flat}
The metric $G^{a,b}$ on open curves is flat. Geodesics are the preimages under the $R$-transform of geo\-de\-sics on the flat space $\on{im} R$.

A path of curves $q: \R \times [0,2\pi] \to C^{a,b}$ in $\on{im} R$ is a geodesic if for each $\theta \in [0,2\pi]$ the curve $t \mapsto q(t,\theta)$ is a geodesic in $C^{a,b}$.

The geodesic distance between $c, \ol c \in \Imm([0,2\pi],\R^2)/\on{Tra}$ is given by the integral over the pointwise distances,
\[ \int_0^{2\pi} \on{dist}( R(c)(\theta), R(\ol c)(\theta)) d\theta\,.\]
However, the minimum over $k \in \mathbb Z$ is not to be taken pointwise, but only once for all values of $\theta$. This corresponds to choosing a continuous lift of the curve $R(c)$ via $\Phi$.
\end{thm}

\begin{proof}
Since the cone $C^{a,b}$ is flat in the sense of Riemannian geometry, so is the space $C^\infty([0,2\pi], C^{a,b})$ of curves in $C^{a,b}$ with respect to the $L^2$-metric given by
\[ G_q(h,k) = \int_0^{2\pi} \langle h, k \rangle\, d\theta\,,\]
for $q \in C^\infty([0,2\pi], C^{a,b})$ and $h, k$ tangent vectors at $q$. Note that this metric does not depend on the basepoint $q$. The metric is the same at all points in the space. Note also that the image of the $R$-transform
\[
\on{im} R = C^\infty([0,2\pi], C^{a,b})\,,
\]
equals the set of curves that lie in the cone $C^{a,b}$. It is a property of the $L^2$-metric that geodesics in $\on{im} R$ are given by paths of curves $q(t,\theta)$, such that for each fixed $\theta$, the curve $q(\cdot, \theta)$ is a geodesic in the cone $C^{a,b}$. The length of this geodesic will be given by an expression of the form \eqref{eq:geod_length} with some $k \in \mathbb Z$, not necessarily the smallest one. The length of the path $q(t,\theta)$ in $\on{im} R$ is given by the integral over the lengths of each point-wise path $q(\cdot, \theta)$. Since we have a continuous family of geodesics on $C^{a,b}$, the value for $k$ will be the same for all $\theta \in [0,2\pi]$. Hence, the geodesic distance between two elements $q, \ol q \in \on{im} R$ is given by the integral
\[ \int_0^{2\pi} \on{dist}( q(\theta), \ol q(\theta) d\theta\,, \]
where the minimum over $k \in \mathbb Z$ is taken only once for all values of $\theta$.
\end{proof}

See \cite[Theorem 9.1]{Ebin1970} for more details on the flat $L^2$-metric.

\section{The $R$-transform for closed curves\label{sec:closedR}}

In this section we want to consider the $R$-transform acting on closed curves. First note that Theorem \ref{thm:grmetric} remains valid, if we replace open curves by closed ones.

Restricting our attention to closed curves $C^{\infty}(S^1, \R^2)$ means that we impose additional constraints on the image of the $R$-transform. From the inversion formula
\[ R\i(q)(\theta) = \frac{1}{a^2} \int_0^\theta \sqrt{q_1(\theta)^2 + q_2(\theta)^2} \begin{pmatrix} q_1(\theta) \\ q_2(\theta) \end{pmatrix} d\theta\,, \]
we see that a curve $q$ is the image of a closed curve only if the condition
\[ F(q) = \int_0^\theta \sqrt{q_1(\theta)^2 + q_2(\theta)^2} \begin{pmatrix} q_1(\theta) \\ q_2(\theta) \end{pmatrix} d\theta = 0 \]
is satisfied. Let us denote the image of the $R$-transform, restricted to closed curves, by
\[ \mathscr C^{a,b} = \{ q \in C^{\infty}(S^1, C^{a,b})\,:\, F(q) = 0 \}\,. \]
Note the difference between the cone $C^{a,b}$, which is a submanifold of $\R^3$ and the space $\mathscr C^{a,b}$, which is a submanifold in the space of curves.

\begin{thm}
\label{thm:basis_compl}
The image $\mathscr C^{a,b}$ of the manifold of closed curves under the $R$-transform is a codimension 2 submanifold of the flat space $C^{\infty}(S^1, C^{a,b})$.

A basis of the orthogonal complement $\left(T_q \mathscr C^{a,b}\right)^\perp$ is given by the two vectors
\begin{align*}
U_1(q)&=\frac{1}{\sqrt{q_1^2+q_2^2}}\,\begin{pmatrix} 2q_1^2+q_2^2\\ q_1 q_2 \\ 0 \end{pmatrix} +\frac{2}{a}\sqrt{4b^2-a^2}\begin{pmatrix} 0\\ 0 \\ q_1 \end{pmatrix},\\
U_2(q)&=\frac{1}{\sqrt{q_1^2+q_2^2}}\,\begin{pmatrix} q_1q_2\\ q_1^2+2q_2^2 \\ 0 \end{pmatrix} +\frac{2}{a}\sqrt{4b^2-a^2}\begin{pmatrix} 0\\ 0 \\ q_2 \end{pmatrix}.
\end{align*}
\end{thm}

\begin{proof}
A basis of $(T_q \mathscr{C}^{a,b})^\perp$ can be computed by projecting the gradients of the two components of the function $F = (F_1,F_2)$ to the tangent space of $C^{\infty}(S^1, C^{a,b})$. Let $q \in \mathscr C^{a,b}$ be a curve and $h \in C^\infty(S^1, \R^3)$ a tangent vector. Then
\begin{align*}
 D_{q,h} F_1 &= \int^{2\pi}_{0} \frac{q_1h_1+q_2h_2}{\sqrt{q_1^2+q_2^2}}\,q_1+\sqrt{q_1^2+q_2^2}\,h_1\, d\theta\\
 D_{q,h} F_2 &= \int^{2\pi}_{0} \frac{q_1h_1+q_2h_2}{\sqrt{q_1^2+q_2^2}}\,q_2+\sqrt{q_1^2+q_2^2}\,h_2 \,d\theta.
\end{align*}
Thus the two gradients are
\begin{align*}
\on{grad}^{L^2} F_1(q)&=\frac{q_1}{\sqrt{q_1^2+q_2^2}}\,\begin{pmatrix} q_1\\ q_2 \\ 0 \end{pmatrix} +\sqrt{q_1^2+q_2^2}\begin{pmatrix} 1\\ 0 \\ 0 \end{pmatrix},\\
\on{grad}^{L^2} F_2(q)&=\frac{q_2}{\sqrt{q_1^2+q_2^2}}\,\begin{pmatrix} q_1\\ q_2 \\ 0 \end{pmatrix} +\sqrt{q_1^2+q_2^2}\begin{pmatrix} 0\\ 1 \\ 0 \end{pmatrix}.
\end{align*}
Differentiating the governing equation for the cone \eqref{formula_cone} we obtain that the tangent space $T_q C^{\infty}(S^1, C^{a,b})$ is given by all curves $h$, which satisfy the pointwise condition
\begin{equation}
\label{eq:tan_cone}
 (4b^2 - a^2) (q_1 h_1 + q_2 h_2) = a^2 q_3 h_3 \,.
\end{equation}
A projection (not necessarily orthogonal) of the vectors $\on{grad}^{L^2} F_i(q)$ can be found by choosing $\la_i$ such that the vector $\on{grad}^{L^2} F_i(q) + \la_i (0, 0, 1)^T$ satisfies \eqref{eq:tan_cone}. 
A simple computation shows that for the component $F_i$ the right value is $\la_i = 2q_i\frac{\sqrt{4b^2 - a^2}}{a}$. This completes the proof.
\end{proof}

We can now use the Gram-Schmidt procedure to compute an orthonormal basis of $(T_q \mathscr{C}^{a,b})^\perp$, which is better suited for computational purposes. The formulas however do not reveal more structure.

\subsection{The curvature of the space $\mathscr{C}^{2a,a}$\label{sec:curvature}}

The curvature of a Riemannian manifold is a symmetric bilinear form on the space of skew bivectors. The normalized quadratic version is called sectional curvature, which in turn is the Gau\ss{} curvature of the geodesic 2-submanifold spanned by the bivector.

As mentioned previously  the $R$-transform has no third component in the case $b=2a$ and the cone $C^{2a,a}$ reduces to the space $\R^2\setminus0$. 
Let $\widetilde{U_1}(q),\widetilde{U_2}(q)\in C^{\infty}(S^1,\R^2)$ be the orthonormal basis of $(T_q\widetilde{\mathcal{C}}^{2a,a})^\bot$ derived via the Gram-Schmidt procedure from the  basis $U_1(q),U_2(q)$. 
Using this ONB we can express the curvature of the cone $\widetilde{\mathcal{C}}^{2a,a}$. 

To do so we take a constant vector field $q\mapsto (q,h)$ on $C^{\infty}(S^1,\mathbb R^2)$ and 
its orthonormal projection
$$
X_h(q)= h-\langle \widetilde{U_1}(q),h \rangle\widetilde{U_1}(q)
-\langle \widetilde{U_2}(q),h \rangle\widetilde{U_2}(q)\in T_q\mathscr{C}^{2a,a}\,.
$$
Then we take the flat covariant derivative in $\Imm(S^1,\mathbb R^2)$
\begin{align*}
\ol\nabla_{X_h(q)}X_k(q) = D_{q,X_h(q)}X_k(q) = d(X_k(q))(X_h(q))\,.
\end{align*}
The orthonormal projection of this vector field onto $(T_q\mathscr{C}^{2a,a})^\bot$ is then equal to the value of the second fundamental form $S\in (T_q\mathscr{C}^{2a,a})^\bot$, i.e.:
\begin{align*}
S(X_h(q),X_k(q))&=\langle d(X_k(q))(X_h(q)), \widetilde{U_1}(q)\rangle_{L^2}  \widetilde{U_1}(q)\\ &\qquad+\langle d(X_k(q))(X_h(q)), \widetilde{U_2}(q)\rangle_{L^2} \widetilde{U_2}(q)\,. 
\end{align*}
The curvature of $\mathscr{C}^{2a,a}$ at $q$ is then given by the 
Gau{\ss}-equation \cite[26.4]{MichorH}: 
\begin{align*}
&\langle R(X_h(q),X_k(q))X_k(q),X_h(q) \rangle_{L^2} 
\\&\qquad
= -\|S(X_h(q),X_k(q))\|^2_{L^2}
+\langle S(X_k(q),X_k(q)),S(X_h(q),X_h(q)) \rangle_{L^2}\,.
\end{align*}

\section{The induced metric on shape space $\mathcal S$\label{sec:shape_space}}

In the rest of the paper we will mainly consider closed curves. However the results can be easily reformulated for the case of open curves.
 
The shape space $\mathcal S$ denotes the space of unparameterized plane curves, which can be represented as the quotient $\mathcal S:=\Imm(S^1,\R^2)/\Diff(S^1)$ of parameterized curves modulo parameterizations. Associated to this quotient is the natural projection
$\pi: \Imm(S^1,\R^2)\to  \Imm(S^1,\R^2)/\Diff(S^1)\,,$
which maps a curve $c$ to its image $C = \pi(c)$.

Given a reparameterization invariant metric $G$ on $\Imm(S^1, \R^2)$ there exists a unique Riemannian metric $\ol G$ on shape space, such that the projection $\pi$ is a Riemannian submersion. Associated to the projection $\pi$ is the decomposition of the tangent bundle $T\Imm(S^1,\R^2)$ into horizontal and vertical parts. The vertical bundle $\on{Ver}$ is the kernel of the projection $\pi$, i.e., $\on{Ver} = \on{ker} T\pi$, and the horizontal bundle 
$\on{Hor}(c) = \on{Ver}(c)^\perp \subset T_c \Imm(S^1,\R^2)$
is defined as the orthogonal complement of $\on{Ver}$ with respect to the Riemannian metric $G$. The action of $\on{Diff}(S^1)$ on $\Imm(S^1, \R^2)$ induces an infinitesimal action of its Lie algebra $\X(S^1)$, given by
\[\ze_\mu(c) = c'\mu \in T_c \Imm(S^1, \R^2)\] 
with $\mu \in \X(S^1)$. This defines a vector field $\ze_\mu$ on the space $\Imm(S^1, \R^2)$. The vertical bundle consists of the image of all infinitesimal vector fields
\[ \on{Ver}(c) = \left\{ \ze_\mu(c)\,:\, \mu \in \X(S^1) \right\}\,. \]
When we fix a curve $c$, there is a one-to-one correspondence between the vector fields on the circle $\X(S^1)$ and the space $\on{Ver}(c)$ given by the map $\mu \mapsto \ze_\mu(c)$.

From the theory of Riemannian submersions \cite{MichorH} it follows that:
\begin{itemize}
\item Geodesics on shape space $\mathcal S$ with respect to $\ol G$ correspond to horizontal geodesics on the manifold $\Imm(S^1,\R^2)$ of parameterized curves with respect to $G^{a,b}$. Horizontal geodesics on $\Imm(S^1, \R^2)$ are those, whose tangent vector lies in the horizontal bundle, i.e. ,$\p_t c(t) \in \on{Hor}(c(t))$.
\item The geodesic distance on shape space can be computed using the formula
\[ \dist(C_0,C_1)= \underset{\ph\in\Diff(S^1)}{\on{inf}}\dist(c_0,c_1\circ\ph),\]
where $\dist$ on the right hand side denotes the geodesic distance on the space $\Imm(S^1,\R^2)$ with respect to the $G^{a,b}$-metric.
\item The curvature of the shape space can be calculated using O'Neil's curvature formula, see for example  \cite{MichorH}. 
Given two orthonormal vector fields $X,Y$ on the space $\mathcal S$ of unparameterized curves, the sectional curvature $K$ is given by
\[
K_{\mathcal S}(X,Y)=K_{\Imm}(\wt X,\wt Y)+\frac34|[\wt X,\wt Y]^{\on{vert}}|^2\,.
\]
Here $\wt X, \wt Y$ are horizontal lifts of the vector fields $X,Y$ to $\Imm(S^1,\R^2)$ and $[\wt X,\wt Y]^{\on{vert}}$ denotes
the vertical projection of the vector field $[\wt X,\wt Y]$.
\end{itemize}
We  can now apply this construction to the metric $G^{a,b}$ and the projection $\pi: \Imm(S^1,\R^2)/\on{Tra}\to \mathcal S/\on{Tra}$
\begin{thm}
Consider the shape space of open curves modulo translations  with the metric that is induces from the $G^{a,b}$-metric on parameterized curves. The curvature of this space is non-negative.
\end{thm}
\begin{proof}
Theorem \ref{thm:open_flat} shows that the space $\Imm([0,2\pi],\R^2)/\on{Tra}$ of parameterized curves is flat, which implies $K_{\Imm}(\wt X, \wt Y) = 0$ in O'Neil's curvature formula. The only remaining term is clearly non-negative:
$
K(X,Y)=\frac34|[\wt X,\wt Y]^{\on{vert}}|^2\,.
$
\end{proof}

\section{Numerical computation of geodesics for parameterized curves}
\label{sec:numerics}

In this section we will describe a way to numerically compute the shortest path  between two parameterized closed curves $c_0$ and $c_1$. We will consider the special case of the metric where $4b^2=a^2=1$. In this case the $R$-transform maps plane curves into plane curves,
\begin{align*}
&R:\;\Imm([0,2\pi],\mathbb R^2)/\on{Tra}\to  C^\infty([0,2\pi],\mathbb R^2)\,,\\ &R(c) = |c'|^{1/2} v\,.
\end{align*}
The cone $C^{a, a/2}$ regarded as a subset of $\R^2$ simplifies to $C^{a, a/2} = \R^2 \setminus 0$. Thus the image of the $R$-transform of open curves is the set
$\on{im} R = C^\infty([0,2\pi], \R^2\setminus 0)$
of all curves that avoid the origin, which is an open subset of the space of all curves. Restricting ourselves to closed curves, we obtain from Theorem \ref{thm:basis_compl} that the image $\mathscr C$ of the $R$-transform is a codimension 2 submanifold of $C^\infty(S^1, \R^2)$. 
\begin{figure*}
\begin{center}
\includegraphics[width=\textwidth]{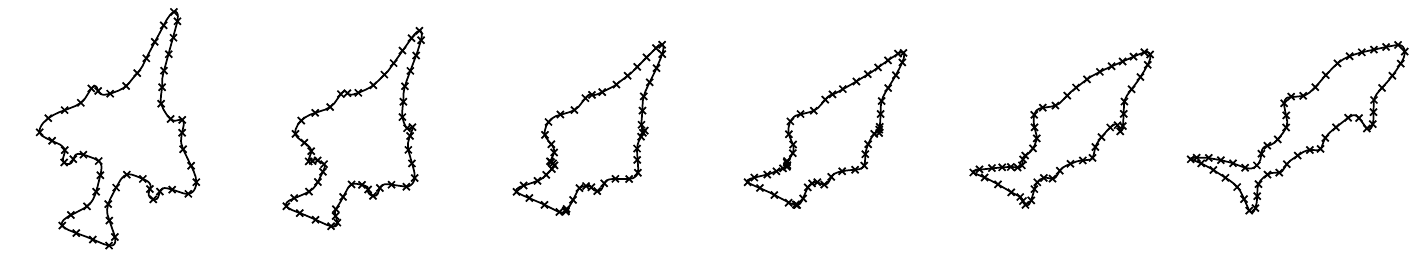}
\end{center}
\caption{Example of a minimal geodesic between two parameterized shapes.}
\label{fig:geoparam}
\end{figure*}

Geodesics in the space $\Imm(S^1, \R^2)/\on{Tra}$ correspond under the $R$-transform to geodesics on $\mathscr C$ with the Riemannian metric induced by the flat $L^2$-metric on $C^\infty(S^1, \R^2)$.
Using Theorem \ref{thm:basis_compl} we can implement a projection operator which, given a curve $q \in \mathscr C$ and a tangent vector $p$ to $q$, computes its orthogonal projection onto $T_q \mathscr C$.

\begin{algorithmic}
\Function{Proj}{$q, p$}
\State $U_1, U_2 \gets \text{ ONB of } (T_q \mathscr C)^\perp$
\State \textbf{return} $p - \langle p, U_1 \rangle\, U_1 - \langle p, U_2 \rangle\, U_2$
\EndFunction
\end{algorithmic}

The forward computation of the geodesic starting from a curve $q \in \mathscr C$ with initial velocity $p \in T_q \mathscr C$ can be seen as a constrained optimization problem. A geodesic on $\mathscr C$ is a minimum of
\[ 
\int_0^1 \int_{S^1} \frac 12 |u(t)|^2 + \langle p(t), \dot q(t) - u(t) \rangle\,d\th  + \la(t) \cdot F(q(t))\, dt\;.
\]
Since we use the Euclidean metric on the ambient space, we will not distinguish between velocities $u(t)$ and momenta $p(t)$. The variable $\la(t) \in \R^2$ is a Lagrange multiplier, which enforces the constraint $F(q)=0$ with
\begin{equation}
\label{closed_constr}
F(q) = \int_{S^1} |q(\theta)| q(\theta) \,d\theta\,,
\end{equation}
and the cone $\mathscr C = F\i(0)$. These are holonomic constraints and there exist several methods, that can solve this problem; see, e.g.,~\cite{LeimkuhlerReich2004} for an overview. We have chosen to use the RATTLE algorithm~\cite{Andersen1983}, which is a symplectic integrator that preserves the constraints through an explicit projection of the velocity vector, forcing it to lie on the constraint surface. The variational equations in continuous time have the form
\begin{align*}
\dot q(t) &= p(t) \\
\dot p(t) &= \la(t) \cdot \nabla F(q(t))\,.
\end{align*}

We denote by $N$ the number of time-steps used by the integrator and set $\De t = \tfrac 1N$.

\begin{algorithmic}
\Function{Exp}{$q, p, N$}
\State $q_0, p_0 \gets q, p$
\For {$i \gets 0,\ldots,N-1$}
\State $\ol p \gets p_i + \tfrac {\De t}{2} \la \cdot \nabla F(q_i)$
\State $q_{i+1} \gets q_i + \De t\, \ol p$
\State estimate initial value for $\la$ to enforce constraint
\State iteratively adapt $\la$ so that $q_{i+1} \in \mathscr C$ 
\State $p_{i+1} \gets \ol p + \tfrac {\De t}{2} \mu \cdot \nabla F(q_{i+1})$
\State $\mu$ chosen such that $p_{i+1} \in T_{q_{i+1}} \mathscr C$
\EndFor
\State \textbf{return} $q, p$
\EndFunction
\end{algorithmic}

The next step is to solve the boundary value problem for parameterized shapes. Given two curves $q_0, q_1 \in \mathscr C$ we need to find the initial velocity $p$, such that the endpoint $\on{Exp}_{q_0} p$ is close to $q_1$. We do so via a fixed-point iteration, using the fact that our space $\mathscr C$ is a submanifold of a flat space.

\begin{algorithmic}
\Function{Log}{$q_0,q_1,N$}
\State $p \gets \De t \,  \text{\Call{Proj}{$q_0, q_1-q_0$}}$
\State $\wt q \gets \text{\Call{Exp}{$q_0, p, 1$}}$
\State $p \gets $ \Call{Proj}{$q_0, \wt q - q_0$}
\While{$\left|\text{\Call{Exp}{$q_0, p, N$}} - q_1\right| > \ep}$
\State $\wt q \gets $\Call{Exp}{$q_0, p, N$}
\State $p \gets p + \al\, \text{\Call{Proj}{$q_0, q_1-\wt q$}}$
\EndWhile
\State \textbf{return} $p$
\EndFunction
\end{algorithmic}
To construct the initial guess for $p$ we project the straight line between $q_0$ and $q_1$ to $T_{q_0}\mathscr C$. This would be a valid initial guess. Based on experiments, however we found that computing one step of the forward-shooting algorithm to obtain $\wt q$ and projecting the straight line $\wt q - q_0$ back to $T_{q_0}\mathscr C$ leads to a better initial guess. The iteration consists of computing the endpoint $\wt q = \on{Exp}_{q_0} p$ of the geodesic with  initial velocity $p$ and updating $p$ in the direction $q_1 - \wt q$.

The parameter $\al$ controls the step-size of the iteration. It can be either fixed for the whole minimization or chosen adaptively. In our experiments we chose $\al$ as large as possible, while still ensuring that the distance decreased, compared to the last iteration.

It is also possible to view this iteration as an approximation to a gradient descent for the geodesic distance 
\[ E(p) = \frac 12 \on{dist}(\on{Exp}_{q_0} p, q_1)^2\,. \]
The derivative of this function is
\[ T_p E(\de p) = G_{\on{Exp}_{q_0} p}\left(-\on{Log}_{\on{Exp}_{q_0} p} q_1, T_p \on{Exp}_{q_0}(\de p)\right)\,. \]
If we approximate the logarithm by the straight line, $\on{Log}_{\on{Exp}_{q_0} p} q_1 \approx q_1 - \on{Exp}_{q_0} p$, the differential of the exponential map by the identity, $T_p \on{Exp}_{q_0}(\de p) \approx \de p$ and the parallel transport along the geodesic $t \mapsto \on{Exp}_{q_0}(tp)$ from $\on{Exp}_{q_0} p$ to $q_0$ by the projection to $T_q \mathscr C$, then we obtain
\[ T_p E(\de p) \approx G_{q_0} \left( -\on{Proj}_{q_0}\left(q_1 - \on{Exp}_{q_0} p\right), \de p  \right) \]
and hence the approximation of the gradient is given by
\[ \nabla_p E \approx -\on{Proj}_{q_0}\left(q_1 - \on{Exp}_{q_0} p\right)\,. \]

For the spatial discretization we replaced the curve $q \in \mathscr C$ by a finite number of points $(q^j)_{1\leq j \leq n}$. Spatially discrete geodesics correspond to minima of
\[
\int_0^1 \sum_{j=1}^n \tfrac 12|u^j(t)|^2 + \langle p^j(t), \dot q^j(t) - u^j(t) \rangle + \la(t) \cdot F(q^j)\,dt.
\]
This variational principle now describes a finite-dimensional Hamiltonian system and we can apply the above algorithms without changes to compute geodesics.

This approach to discretization works for all metrics that can be represented via a transform as in  Section \ref{sec:relating}, provided we can characterize the image of this transform, similarly to \eqref{closed_constr}.
\begin{table}
\begin{center}
\begin{tabular}{|cc|cc|cc|}
 && \multicolumn{2}{|c|}{I$_1 \to $ I$_2$} & \multicolumn{2}{|c|}{I$_2 \to $ I$_1$} \\
I$_1$& I$_2$&  distance & time (s) & distance& time (s)  \\
cat &cow & 17.106 & 1.93 & 17.106 & 2.25 \\
cat &dog & 21.349 & 2.59 & 21.349 & 2.24 \\
cat &donkey & 25.273 & 2.59 & 25.273 & 2.59  \\
cow &dog  & 18.389 & 2.22 & 18.389 & 2.21  \\
cow &donkey & 20.206 & 2.56 & 20.206 & 2.26  \\
dog &donkey & 14.983 & 2.25 & 14.983 & 1.91\\
shark&airplane & 20.488 & 2.65 & 20.488 & 2.34 \\
\end{tabular}
\end{center}
\caption{Geodesic distances between parameterized shapes, together with the time in seconds to compute the distance. It can be seen the the distances are symmetric between pairs of parameterized shapes.}
\label{tab:distance}
\end{table}
\subsection{Numerical Results\label{sec:num1}}

In this section we present experiments demonstrating our numerical computations with parameterized curves. The algorithms described in the previous section were implemented in Python using the NumPy and SciPy libraries. The shapes used for the experiments come from the database of closed binary shapes collected by the LEMS Vision Group at Brown university (\url{http://www.lems.brown.edu/~dmc}). Each curve was initially parameterized by a set of 300 points positioned equidistantly along the curve, and each shape is centered at an arbitrary point (the position of the shape does not matter, as the metric is invariant under translations). The geodesics were computed using $N=25$ time steps. 

Figure~\ref{fig:geoparam} shows a set of shapes along the minimal geodesic path between two pairs of shapes. The geodesics are sampled at timesteps 0, 5, 10, 15, 20 and 25. The first and last images in each row thus show the template and target curves respectively. The template curve is always parameterized proportional to arc length. Table~\ref{tab:distance} shows the geodesic distances forwards and backwards between parameterized shapes, along with the time for the computation (based on a Python implementation running on a virtual server with 3.6Gb of memory and access to a dual core 3GHz Xeon processor).  It can be seen that these distances are symmetric, as expected. 

\begin{figure*}
\begin{center}
\includegraphics[width=0.9\textwidth]{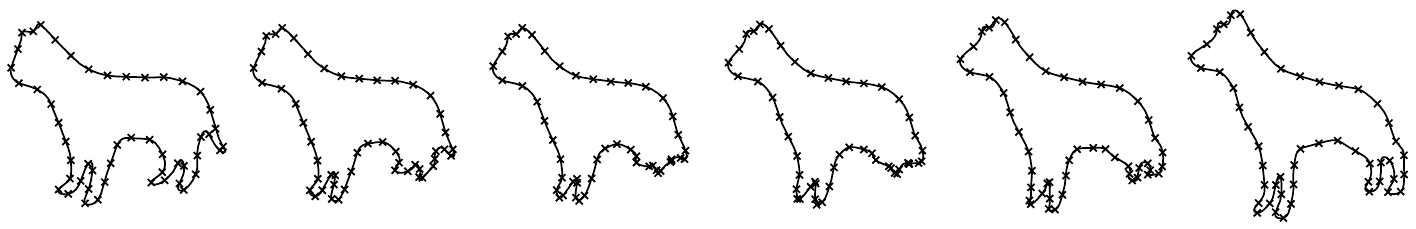}
\includegraphics[width=0.9\textwidth]{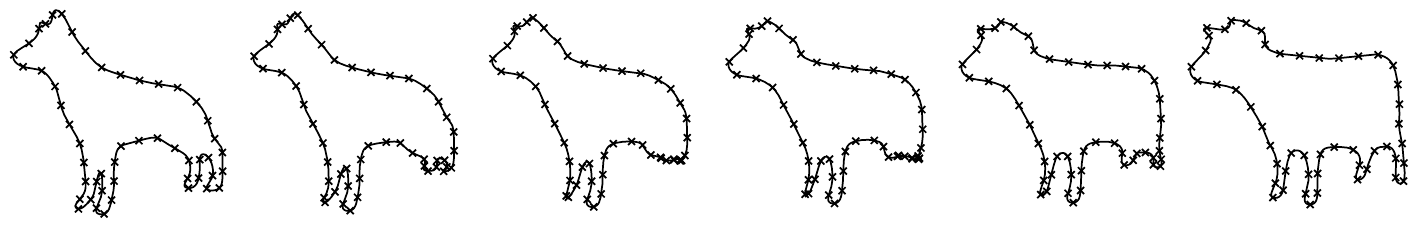}
\includegraphics[width=0.9\textwidth]{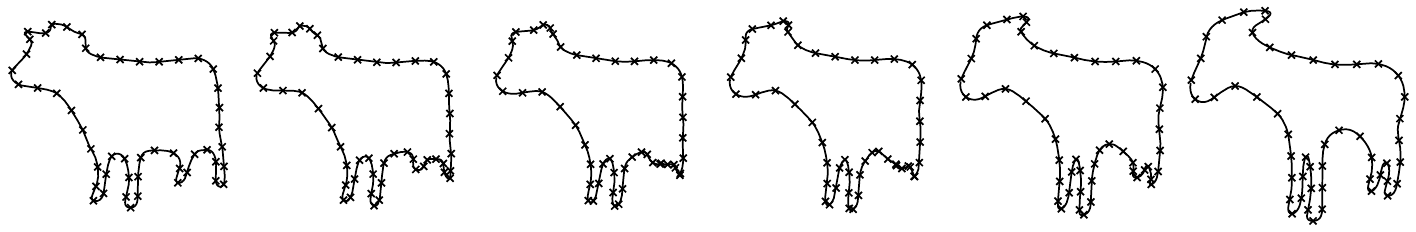}
\includegraphics[width=0.9\textwidth]{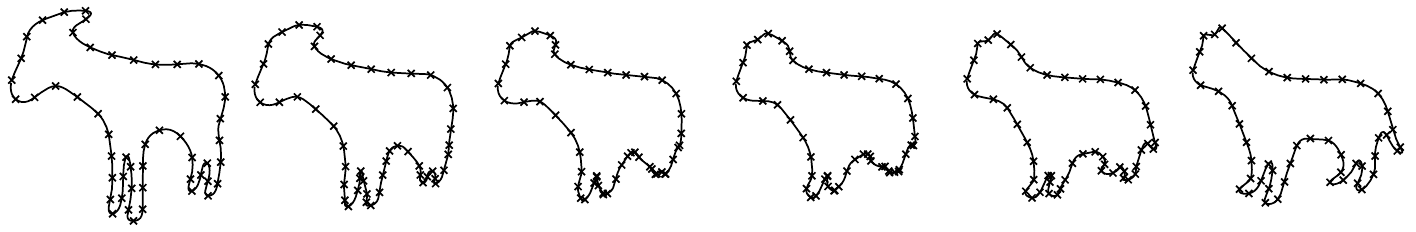}
\end{center}
\caption{Examples of geodesics between several parametrized shapes. The geodesics shown are from the cat to the dog in the first row, from the dog to the cow in the second, from the cow to the donkey in the third and from the donkey back to the cat in the last row thus forming a geodesic quadrangle.}
\label{fig:quadrangle}
\end{figure*}

\section{Numerical Computation of geodesics on unparameterized curves\label{sec:unparam}}
\subsection{Finding the optimal parameterization}
\label{sec:num_shape}

In this section we build on the previous section and describe our method for finding geodesics between unparameterized shapes. As was explained in Section \ref{sec:shape_space},  this corresponds to finding horizontal geodesics or equivalently finding the minimum of
\[ \dist(C,D)= \underset{\ps\in\Diff(S^1)}{\on{inf}}\dist(c,d\circ\ps),\]
with $C = \pi(c)$, $D = \pi(d)$ and where $\on{dist}$ on the right denotes the geodesic distance on parameterized curves, computed as described in Section \ref{sec:numerics}. We will compute this minimum using a gradient descent algorithm on
\begin{equation}
\label{eq:en_phi}
 E(\ph) = \frac 12 \on{dist}(c, d\circ\ph\i)^2\,.
\end{equation}
We use $d\circ\ph\i$ instead of $d\circ\ph$, in order to have a left action of $\on{Diff}(S^1)$ on the space of curves. We want to compute the right-trivialized gradient $\nabla_\ph E$ of $E$, defined as
$\langle \nabla_\ph E, \mu \rangle_{\X(S^1)} = T_\ph E(\mu \circ \ph)\,$
for $\mu \in \X(S^1)$ and a choice of an inner product $\langle\cdot,\cdot\rangle_{\X(S^1)}$ on $\X(S^1)$.

The gradient of
\[ G(p) = \frac 12 \on{dist}(p, q)^2 \]
is given by
$\nabla_p G = -\on{Log}_p q\,.$
Hence
\[ T_\ph E(\de \ph) = G_{d\circ\ph\i}\left(-\on{Log}_{d\circ\ph\i} c,  -(d\circ\ph\i)'\, \de\ph\circ\ph\i\right) \]
and by writing $\de \ph = \mu \circ\ph$ for some $\mu \in \mf X(S^1)$ we get
\[ T_\ph E(\mu \circ \ph) = G_{d\circ\ph\i}(\on{Log}_{d\circ\ph\i} c,  (d\circ\ph\i)' u)\,. \]
After introducing an inner product on $\mf X(S^1)$, we could compute the gradient of $E(\ph)$ by solving
\[ \langle \nabla_\ph E, u \rangle_{\mf X(S^1)} = G_{d\circ\ph\i}(\on{Log}_{d\circ\ph\i} c,  (d\circ\ph\i)' u)\,. \]
There is however a better expression for the gradient, which fits better with the action of $\on{Diff}(S^1)$ on $\Imm(S^1,\R^2)$.
Consider the related function
\[ \ol E(\ph) = \frac 12 \on{dist}(c\circ\ph\i, d)^2\,. \]
Its gradient can be computed using the invariance of the metric under $\ph$, i.e., $\ol E(\ph) = E(\ph\i)$, which implies
\begin{align*}
T_\ph \ol E(\de \ph)
&= T_{\ph\i} E\left(-(\ph\i)' \de\ph\circ\ph\i\right) \\
&= -T_{\ph\i} E\left( \left(\frac 1{\ph'} \de\ph\right) \circ\ph\i\right) \\
&= G_{d\circ\ph}\left(-\on{Log}_{d\circ\ph} c, (d\circ\ph)'\,\frac 1{\ph'} \de\ph \right) \\
&= G_{d\circ\ph}\left(-\on{Log}_{d\circ\ph} c, (d'\circ\ph)\, \de\ph \right)\,.
\end{align*}
Using the invariance of $\on{Log}$ and the metric $G_c$ under reparameterizations,
\begin{align*}
 \on{Log}_{d\circ\ph} c\circ\ph &= \left(\on{Log}_d c\right)\circ\ph \\
 G_{d\circ\ph}(h\circ\ph, k\circ\ph) &= G_d(h, k)\,,
\end{align*}
we get
\begin{align*}
T_\ph \ol E(\de \ph)
&= G_{d\circ\ph}\left(-\left(\on{Log}_{d} (c\circ\ph\i)\right)\circ\ph, (d'\circ\ph)\, \de\ph \right) \\
&= G_{d}\left(-\on{Log}_{d} (c\circ\ph\i), d'(\de\ph\circ\ph\i) \right)\,,
\end{align*}
and hence the gradient can be obtained by solving
\[ \langle \nabla_\ph \ol E, \mu \rangle_{\mf X(S^1)} = -G_{d}(\on{Log}_{d}(c\circ\ph\i),  d' \mu)\,. \]
Since the functions $E$ and $\ol E$ differ only by exchanging $c$ and $d$, we can also express the gradient of $E$ by:
\begin{equation}
\label{grad_expr2}
\langle \nabla_\ph E, \mu \rangle_{\mf X(S^1)} = -G_{c}(\on{Log}_{c}(d\circ\ph\i),  c' \mu)\,.
\end{equation}

As the inner product on the space of vector fields we can use the one induced by the identification of $\X(S^1)$ with the vertical space at the curve $c$ using the infinitesimal action $\ze_\mu(c) = c'\mu$. The inner product is thus given by
\begin{equation}
\label{induced_metric_s1}
 \langle \mu, \nu \rangle_{\mf X(S^1)} = G_c(c'\mu, c'\nu)\,.
\end{equation}

\begin{thm}
\label{thm:grad}
The (right-trivialized) gradient of the energy
\[ E(\ph) = \frac 12 \on{dist}(c, d\circ\ph\i)^2 \]
with respect to the inner product \eqref{induced_metric_s1} is given by the vector field on $S^1$ corresponding to the vertical projection of $-\on{Log}_c(d\circ\ph\i)$, i.e.
\[ c' \nabla_\ph E = -\on{Ver}_c\left( \on{Log}_c(d\circ\ph\i)\right)\,. \]
It can be computed by solving the equation
\[ G_c(c' \nabla_\ph E, c' \mu) = -G_{c}(\on{Log}_{c} (d\circ\ph\i),  c' \mu)\quad \forall \mu \in \mf X(S^1) \]
for $\nabla_\ph E \in \mf X(S^1)$.
\end{thm}

\begin{proof}
Note that the equation
\[ G_c(h, c' \mu) = -G_{c}(\on{Log}_{c} d\circ\ph\i,  c' \mu) \]
is satisfied for all $\mu \in \mf X(S^1)$ if and only if 
\[ h = -\on{Ver}_c\left( \on{Log}_c(d\circ\ph\i)\right)\,. \]
Each element in the vertical space corresponds via the infinitesimal action to one element of $\mf X(S^1)$. The theorem now follows from combining equations \eqref{grad_expr2} and \eqref{induced_metric_s1}.
\end{proof}

The next algorithm computes the element $\mu \in \X(S^1)$ corresponding to the orthogonal projection of a vector $h \in T_c \Imm(S^1, \R^2)$ to the vertical subspace $\on{Ver}(c)$.

\begin{algorithmic}
\Function{Ver}{$c, h$}
\State $\nu \gets \text{ test function on } \X(S^1)$
\State $\mu \gets \text{ solution of } G_c(c'\mu, c'\nu) = G_c(h, c'\nu)$
\State \textbf{return} $\mu$
\EndFunction
\end{algorithmic}
We use a finite element method with Lagrange elements of first order to numerically compute the vertical projection. A more explicit formula for the inner product $G_c(h,k)$ is given by
\[ G_c(h,k) = \int_{S^1} \frac{\langle h',k'\rangle}{|c'|} - \frac 34 \frac{\langle h', c'\rangle \langle k', c' \rangle}{|c'|^3} \,d\theta\,. \]

The algorithm to find geodesics between unparameterized shapes $C,D$ takes as input two parameterizations $c,d$ of these shapes such that $\pi(c)=C$ and $\pi(d)=D$ and finds the diffeomorphism $\ps \in \on{Diff}(S^1)$, such that the geodesic distance $\on{dist}(c, d\circ\ps)$ is minimal.

\begin{algorithmic}
\Function{SolveBVP}{$c, d,N$}
\State $\ps \gets \on{Id}_{S^1}$ \Comment Notation $\ps:=\ph\i$ in \eqref{eq:en_phi}
\While{$\on{dist}(c, d\circ\ps)$ is not minimal}
\State $h \gets $\Call{Log}{$c, d\circ\ps, N$}
\State $\mu \gets $\Call{Ver}{$c, h$} \Comment Notation $\mu := -\nabla_\ph E$
\State $\et \gets $\Call{Flow}{$-\mu,\on{Id}_{S^1},\al$}
\State $\ps \gets \ps\circ\et$
\EndWhile
\State \textbf{return} $\ps$
\EndFunction
\end{algorithmic}
To understand the algorithm note that $\ps$ corresponds to $\ph\i$ in Theorem \ref{thm:grad}. There is no need to compute $\ph$ itself, as only $\ph\i$ is necessary to compute the reparameterization of the curve $d$. In each iteration of the algorithm we first compute the gradient $\mu = -\nabla_\ph E$ with the help of Theorem \ref{thm:grad}. A continuous gradient descent would take the form
$\p_t \ph = -\nabla_\ph E \circ \ph\,$, while
a first order time-discretization would be
\[ \ph_{i+1} = \on{Fl}^\mu(\al,\on{Id}_{S^1})\circ\ph_i\,  \]
with $\on{Fl}^\mu(\al,\on{Id}_{S^1})$ denoting the flow of the vector field $\mu=-\nabla_\ph E$ up to time $\al$ starting from $\Id_{S^1}$ at time $0$. Since we are only interested in $\ps=\ph\i$, we can rewrite this as
\begin{align*}
\ps_{i+1} = \ph_{i+1}\i &= \ph_i\i \circ \on{Fl}^\mu(\al,\on{Id}_{S^1})\i \\
&= \ps_i \circ \on{Fl}^{-\mu}(\al,\on{Id}_{S^1})\,.
\end{align*}
The gradient descent step is repeated until the relative decrease of $\on{dist}(c,d\circ\ps)$ falls below a prescribed threshold.
Similarly to the function $\on{Log}_{q_0}(q_1)$ in Section \ref{sec:numerics}, the parameter $\al$ is the step-size and is chosen adaptively to ensure the distance decreases in each step.

\subsection{Adaptive grid refinement\label{sec:add}}

The behaviour of shortest paths between unparameterized curves can be understood as a combination of stretching and bending. Bending of a parameterized curve is numerically well behaved, since uniformly sampled curves will stay approximately uniformly sampled. However, stretching  tends to expand a very short section of the curve into a much larger one and will lead to a curve that is under-sampled in the expanded area, unless the original parameterization was chosen to counter this effect. See Figure~\ref{fig:add} for an example of this. 

In the gradient descent algorithm from Section \ref{sec:num_shape} we start with a curve that is sampled uniformly at the points 
\[ x_0 = 0, x_1 = \frac{2\pi}{n},\ldots, x_{n-1} = \frac{n-1}{n}2\pi\,. \]
At each iteration, given a grid $x_0,\ldots,x_{n-1}$ for the curve $c$, the curve $d$ would be sampled at the points $\ps(x_0),\ldots,\ps(x_{n-1})$. Stretching of the curve $c$ would correspond to a large derivative $\ps'(x)$ or a large distance $|\ps(x_{i+1}) - \ps(x_i)|$ between two consecutive points. We add points to the grid, whenever the distance exceeds that of a uniform grid, i.e.
\[ |\ps(x_{i+1}) - \ps(x_i)| > \frac{2\pi}{n_0}\,. \]
Here $n_0$ is the size of the original grid. By this we ensure that the target curve $d\circ\ps$ will not be under-sampled.

Similarly, we remove a point $x_i$ whenever the neighbouring points would be sufficient to provide enough resolution, i.e., these conditions are satisfied:
\[ |x_{i+1} - x_{i-1}| < \frac{2\pi}{n_0} \text{ and } |\ps(x_{i+1}) - \ps(x_{i-1})| < \frac{2\pi}{n_0} \:\:. \]

\subsection{Numerical Results}
\label{sec:results}

In order to compute the vertical projection of a tangent vector as described in Section~\ref{sec:num_shape} we used the finite element library FEniCS \cite{LoggWells10}. To find the optimal parameterization the gradient descent algorithm from Section~\ref{sec:num_shape} with a maximum of 100 iterations was used.

Figure~\ref{fig:results} shows examples of some of the geodesics. From a mathematical point of view we would expect the geodesics between two shapes to be symmetric, i.e., interchanging the curves $c_0$ and $c_1$, representing the shapes, should have no effect.  However, if we look at the second and third rows of the figure, we see that the path between the cow and the dog is not completely symmetrical. This behaviour originates in the incompletness of the metric. The optimal geodesic from the dog to the cow (shown backwards in the last line) shrinks a leg of the dog and at the same time grows another leg of the cow from the body, and such growth tends to originate from a single point. Thus, one point of the dog expands to create the whole leg of the cow, while the whole leg of the dog collapses to one point on the cow. 

In Figure~\ref{fig:add} we show a more pronounced example of this behaviour. The template shape $c_0$ is an ellipse and the target shape $c_1$ is an ellipse with the added fold both discretized by equidistantly spaced points. There is no problem to compute the geodesic between $c_0$ and $c_1$. However in the process of computing the geodesic between the shapes $\pi(c_0)$ and $\pi(c_1)$ we see that the horizontal geodesic wants to create the whole fold out of a single point on the ellipse. The development of this singularity means that there exists no geodesic in $\mathcal S$ between the shapes $\pi(c_0)$ and $\pi(c_1)$. This behaviour was also described in \cite[Sect. 4.2]{Michor111} for a related metric and it is a property of the Riemannian metric, rather than an artefact of the discretization.

The incompleteness of the metric has consequences for its use in shape analysis. If geodesics don't necessarily exist between any two shapes, then it is not possible to linearize shape space using the Riemannian exponential map. Statistical methods that are based on the exponential map will be blind to potentially large parts of shape space.

\begin{figure}
\begin{center}
\includegraphics[width=.24\textwidth]{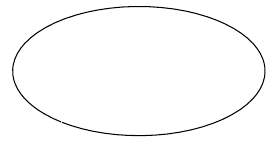}
\includegraphics[width=.24\textwidth]{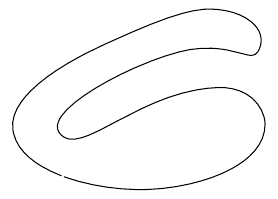}
\includegraphics[width=.24\textwidth]{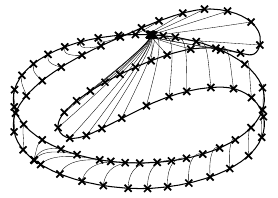}
\includegraphics[width=.24\textwidth]{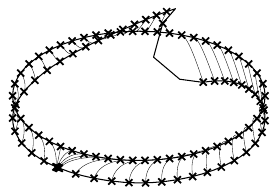}
\end{center}
\caption{This image shows the minimal geodesic between the two shapes in the first row. The template shape is an ellipse while the target shape is an ellipse with a large fold. The bottom left image shows the minimal geodesic computed with the grid adaptively refined as in Section~\ref{sec:add}. In the bottom right image the geodesic is computed without grid refinement, which leads to a loss of resolution along the fold. Since the fold is growing out of a point it is necessary to refine the grid in the neighbourhood of this point to accurately capture the fold.}
\label{fig:add}
\end{figure}

\begin{figure*}
\begin{center}
\includegraphics[width=\textwidth]{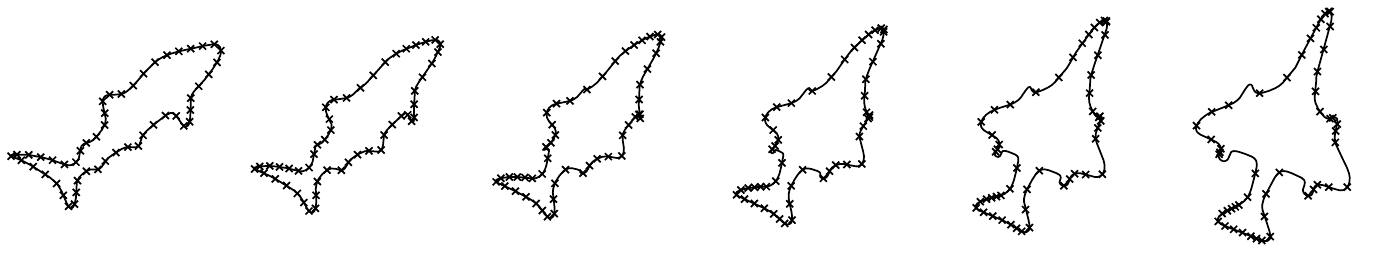}
\includegraphics[width=\textwidth]{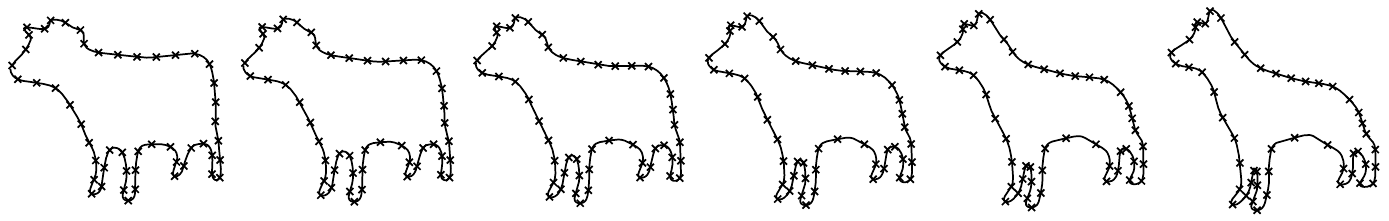}
\includegraphics[width=\textwidth]{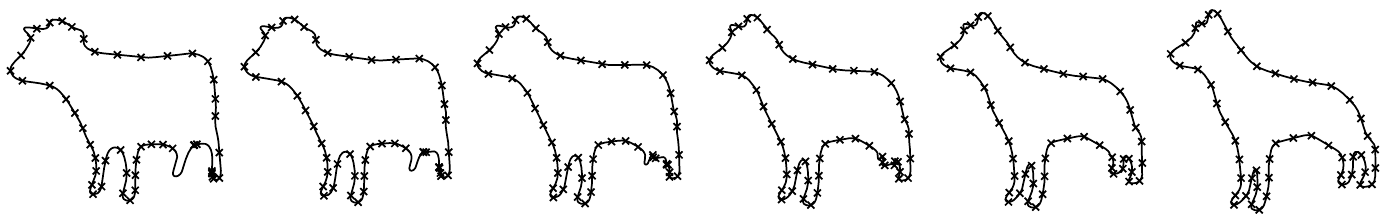}
\end{center}
\caption{Minimal geodesics between various shapes with optimized parameterization. The bottom row shows the backwards geodesic to the middle row (template and target curves reversed) but printed backwards, so that it can be more easily compared with the line above. Mathematically this should be perfectly symmetric, but since there is stretching and compression along the geodesics it is not, and the geodesic distances computed will differ.}
\label{fig:results}
\end{figure*}

\section{Conclusions}
\label{sec:conclusions}
Riemannian metrics on shape space of plane curves are of great interest in a wide variety of applications in image  analysis and computational anatomy. Due to the infinite-dimensional nature of shape space, metrics that allow for efficient computations of geodesics and distance between shapes are particularly useful.

In this paper we generalize the $R$-transform, first introduced in the work of \cite{Srivastava2011}, to the class of all elastic metrics whose coefficients satisfy $4b^2\geq a^2$, where $a$ and $b$ are parameters controlling the degree of bending and stretching of the curve respectively (see Section~\ref{sec:lit}). This transformation allows us to obtain efficient algorithms for computing the geodesic distance between shapes as well as to gain a better understanding of the geometry of the space. For the case of open curves we show that the space of parameterized open curves is a flat space in the sense of Riemannian geometry and obtain explicit formulas for geodesics. As a consequence, it follows that the shape space of unparameterized open curves has positive sectional curvature.
For closed curves the situation is more difficult, since the space of parameterized curves is not flat anymore, but the representation is still very useful from a numerical point of view.
 
We have presented numerical results for geodesics between both parametrized and unparameterized closed curves. For parameterized curves the computed metric is symmetric (that is, the distance between two curves $c$ and $d$ does not depend which is the source and which the target), while for unparametrized curves incompleteness of the space makes the task of computing geodesics much more difficult. Further work is required to understand this behaviour and to either develop numerical methods capable of dealing with this situation or to find Riemannian metrics for which the geodesic boundary value problem on shape space is solvable. We also plan to demonstrate the use of the metrics for the classification of sets of shapes.

\section*{Acknowledgments}
This research was partly supported by the FWF-Projects P2462511 and P21030-N13 as well as by an Advanced Grant from the European Research Council
and the Royal Society of New Zealand Marsden Fund. We thank Colin Cotter for  helpful discussions and valuable comments.

\bibliographystyle{abbrv}

\begin{thebibliography}{10}

\bibitem{Andersen1983}
H.~C. Andersen.
\newblock Rattle: A “velocity” version of the shake algorithm for molecular
  dynamics calculations.
\newblock {\em Journal of Computational Physics}, 52(1):24 -- 34, 1983.

\bibitem{MBMB2011}
M.~Bauer and M.~Bruveris.
\newblock A new {R}iemannian setting for surface registration.
\newblock In {\em 3nd MICCAI Workshop on Mathematical Foundations of
  Computational Anatomy}, pages 182--194, 2011.

\bibitem{Bauer2013a_preprint}
M.~Bauer, M.~Bruveris, and P.~W. Michor.
\newblock ${R}$-transforms for {S}obolev ${H^2}$-metrics on spaces of plane
  curves.
\newblock {\em \emph{To appear in:} Geometry, Imaging and Computing}, 2013.

\bibitem{Michor119}
M.~Bauer, P.~Harms, and P.~W. Michor.
\newblock Sobolev metrics on shape space of surfaces.
\newblock {\em J. Geom. Mech.}, 3(4):389--438, 2011.

\bibitem{Michor118}
M.~Bauer, P.~Harms, and P.~W. Michor.
\newblock Almost local metrics on shape space of hypersurfaces in n-space.
\newblock {\em SIAM J. Imaging Sci.}, 5:244--310, 2012.

\bibitem{Bauer2012b}
M.~Bauer, P.~Harms, and P.~W. Michor.
\newblock Curvature weighted metrics on shape space of hypersurfaces in
  {$n$}-space.
\newblock {\em Differential Geom. Appl.}, 30(1):33--41, 2012.

\bibitem{Michor125}
M.~Bauer, P.~Harms, and P.~W. Michor.
\newblock Sobolev metrics on shape space, ii: Weighted sobolev metrics and
  almost local metrics.
\newblock {\em J. Geom. Mech.}, 4(4):365 -- 383, 2012.

\bibitem{Charpiat2007}
G.~Charpiat, R.~Keriven, and O.~Faugeras.
\newblock Shape statistics for image segmentation with priors.
\newblock In {\em Conference on Computer Vison and Pattern Recognition}, 2007.

\bibitem{Cotter2012}
C.~{Cotter}, A.~{Clark}, and J.~{Peir\'{o}}.
\newblock A reparameterisation based approach to geodesic constrained solvers
  for curve matching.
\newblock {\em International Journal of Computer Vision}, 99:103--121, 2012.

\bibitem{Ebin1970}
D.~G. Ebin and J.~Marsden.
\newblock Groups of diffeomorphisms and the motion of an incompressible fluid.
\newblock {\em Ann. of Math. (2)}, 92:102--163, 1970.

\bibitem{Glaunes2008}
J.~Glaun\`es, A.~Qiu, M.~Miller, and L.~Younes.
\newblock Large deformation diffeomorphic metric curve mapping.
\newblock {\em International Journal of Computer Vision}, 80:317--336, 2008.

\bibitem{GrMi1998}
U.~Grenander and M.~I. Miller.
\newblock Computational anatomy: An emerging discipline.
\newblock {\em Quart. Appl. Math.}, 56:617--694, 1998.

\bibitem{Kurtek2011b}
S.~Kurtek, E.~Klassen, Z.~Ding, S.~Jacobson, J.~Jacobson, M.~Avison, and
  A.~Srivastava.
\newblock Parameterization-invariant shape comparisons of anatomical surfaces.
\newblock {\em IEEE Transactions on Medical Imaging}, 30(3):849--858, 2011.

\bibitem{Kurtek2011a}
S.~Kurtek, E.~Klassen, J.~Gore, Z.~Ding, and A.~Srivastava.
\newblock Elastic geodesic paths in shape space of parametrized surfaces.
\newblock {\em IEEE Transactions on Pattern Analysis and Machine Intelligence},
  PP(99):1, 2011.

\bibitem{LeimkuhlerReich2004}
B.~Leimkuhler and S.~Reich.
\newblock {\em Simulating {H}amiltonian Dynamics}.
\newblock Number~14 in Cambridge Monographs on Applied and Computational
  Mathematics. Springer, 2004.

\bibitem{LoggWells10}
A.~{Logg} and G.~Wells.
\newblock {DOLFIN}: Automated finite element computing.
\newblock {\em ACM Transactions on Mathematical Software}, 37(2), 2010.

\bibitem{Mani2010}
M.~Mani, S.~Kurtek, C.~Barillot, and A.~Srivastava.
\newblock A comprehensive {R}iemannian framework for the analysis of white
  matter fiber tracts.
\newblock In {\em IEEE Symposium on Biomedical Imaging}, pages 1101 --1104,
  2010.

\bibitem{Mennucci2008}
A.~C.~G. Mennucci, A.~Yezzi, and G.~Sundaramoorthi.
\newblock Properties of sobolev-type metrics in the space of curves.
\newblock {\em Interfaces and Free Boundaries}, 8(4):423--445, 2008.

\bibitem{MichorH}
P.~W. Michor.
\newblock {\em Topics in differential geometry}, volume~93 of {\em Graduate
  Studies in Mathematics}.
\newblock American Mathematical Society, Providence, RI, 2008.

\bibitem{Michor102}
P.~W. Michor and D.~Mumford.
\newblock Vanishing geodesic distance on spaces of submanifolds and
  diffeomorphisms.
\newblock {\em Doc. Math.}, 10:217--245 (electronic), 2005.

\bibitem{Michor98}
P.~W. Michor and D.~Mumford.
\newblock Riemannian geometries on spaces of plane curves.
\newblock {\em J. Eur. Math. Soc. (JEMS) 8 (2006), 1-48}, 2006.

\bibitem{Michor107}
P.~W. Michor and D.~Mumford.
\newblock An overview of the {R}iemannian metrics on spaces of curves using the
  {H}amiltonian approach.
\newblock {\em Appl. Comput. Harmon. Anal.}, 23(1):74--113, 2007.

\bibitem{Mio2007}
W.~Mio, A.~Srivastava, and S.~Joshi.
\newblock On shape of plane elastic curves.
\newblock {\em International Journal of Computer Vision}, 73:307--324, 2007.

\bibitem{Samir2012}
C.~Samir, P.-A. Absil, A.~Srivastava, and E.~Klassen.
\newblock A gradient-descent method for curve fitting on {R}iemannian
  manifolds.
\newblock {\em Found. Comput. Math.}, 12(1):49--73, 2012.

\bibitem{Shah2006}
J.~{Shah}.
\newblock An ${H}^2$ type {R}iemannian metric on the space of planar curves.
\newblock In {\em 1st MICCAI Workshop on Mathematical Foundations of
  Computational Anatomy}, pages 40--46, 2006.

\bibitem{Sharon06}
E.~Sharon and D.~Mumford.
\newblock {2D-Shape Analysis Using Conformal Mapping}.
\newblock {\em International Journal of Computer Vision}, 70(1):55--75, 2006.

\bibitem{Srivastava2011}
A.~Srivastava, E.~Klassen, S.~Joshi, and I.~Jermyn.
\newblock Shape analysis of elastic curves in {E}uclidean spaces.
\newblock {\em IEEE Transactions on Pattern Analysis and Machine Intelligence},
  33(7):1415--1428, 2011.

\bibitem{Sundaramoorthi2011}
G.~{Sundaramoorthi}, A.~{Mennucci}, S.~{Soatto}, and A.~{Yezzi}.
\newblock A new geometric metric in the space of curves, and applications to
  tracking deforming objects by prediction and filtering.
\newblock {\em SIAM Journal on Imaging Sciences}, 4(1):109--145, 2011.

\bibitem{Sundaramoorthi2007}
G.~Sundaramoorthi, A.~Yezzi, and A.~Mennucci.
\newblock Sobolev active contours.
\newblock {\em International Journal of Computer Vision}, 73:345--366, 2007.

\bibitem{Younes1998}
L.~Younes.
\newblock Computable elastic distances between shapes.
\newblock {\em SIAM J. Appl. Math.}, 58(2):565--586 (electronic), 1998.

\bibitem{Michor111}
L.~Younes, P.~W. Michor, J.~Shah, and D.~Mumford.
\newblock A metric on shape space with explicit geodesics.
\newblock {\em Atti Accad. Naz. Lincei Cl. Sci. Fis. Mat. Natur. Rend. Lincei
  (9) Mat. Appl.}, 19(1):25--57, 2008.

\end{thebibliography}

\end{document}